\documentclass[12pt]{amsart}
\usepackage{jabbrv}

\usepackage{Commands}
\allowdisplaybreaks
\numberwithin{equation}{section}


\makeatletter
\newcommand*{\itemequation}[3][]{%
	\item
	\begingroup
	\refstepcounter{equation}%
	\ifx\\#1\\%
	\else
	\label{#1}%
	\fi
	\sbox0{#2}%
	\sbox2{$\displaystyle#3\m@th$}%
	\sbox4{ \@eqnnum}%
	\dimen@=.5\dimexpr\linewidth-\wd2\relax
	\let\CenterInSpace=Y%
	\ifcase
	\ifdim\wd0>\dimen@
	\z@
	\else
	\ifdim\wd4>\dimen@
	\z@
	\else
	\@ne
	\fi
	\fi
	\let\CenterInSpace=N%
	\fi
	\ifdim\dimexpr\wd0+\wd2+\wd4\relax>\linewidth
	\@latex@warning{Equation is too large}%
	\fi
	\noindent
	\rlap{\copy0}%
	\ifx\CenterInSpace Y%
	\rlap{\hbox to \linewidth{\kern\wd0\hss\copy2\hss\kern\wd4}}%
	\else
	\rlap{\hbox to \linewidth{\hfill\copy2\hfill}}%
	\fi
	\hbox to \linewidth{\hfill\copy4}%
	\hspace{0pt}
	\endgroup
	\ignorespaces
}
\makeatother

\providecommand{\phantomsection}{}
\AtBeginDocument{\let\textlabel\label}
\makeatletter
\newcommand{\mylabel}[2]{\raisebox{.7\normalbaselineskip}{\phantomsection}#1%
	\def\@currentlabel{#1}\textlabel{#2}}
\makeatother

\author{Jason Atnip}
\address{School of Mathematics and Statistics, University of New South Wales, Sydney, NSW 2052, AUS}
\email{j.atnip@unsw.edu.au  \hspace*{0.42cm} \it Web: \rm http://atnipmath.com}
\title[Non-Autonomous Meromorphic Functions]{Dimensions of Non-Autonomous Meromorphic Functions of Finite Order}
\date{}
\begin{document}

\begin{abstract}
	In this paper we study two classes of meromorphic functions previously studied by Mayer in \cite{mayer_size_2007} and by Kotus and Urba\'nski in \cite{kotus_hausdorff_2008}. In particular we estimate a lower bound for the Julia set and the set of escaping points for non-autonomous additive and affine perturbations of functions from these classes. For particular classes we are able to calculate these dimensions exactly. In these cases, we are able to reinterpret our results to show that the Hausdorff dimension of the set of points in the Julia set which escape to infinity is stable under sufficiently small additive perturbations. We accomplish this by constructing non-autonomous iterated function systems, whose limit sets sit inside of the aforementioned non-autonomous Julia sets. We also give estimates for the eventual and eventual hyperbolic dimensions of these non-autonomous perturbations. 
\end{abstract}
\maketitle

\section{Introduction}
Much work has been done recently concerning the (autonomous) dynamics of transcendental meromorphic functions. In \cite{mayer_size_2007}, Mayer used infinite iterated function systems to find a lower bound for the Hausdorff dimension of the Julia set of meromorphic functions of finite order as well as their hyperbolic dimension. Previously similar techniques were used by Kotus and Urba\'nski in \cite{kotus_hausdorff_2003} and Roy and Urba\'nski in \cite{roy_random_2011} to find a lower bound for the Hausdorff dimension of the Julia set of autonomous and random systems of elliptic functions respectively by using the theories of infinite autonomous and random iterated function systems. In a similar fashion, we will use the theory of non-autonomous conformal iterated function systems as developed in \cite{rempe-gillen_non-autonomous_2016} and \cite{atnip_nonautonomous_2017}, to find a lower and upper bound for the Hausdorff dimension of the set of points in the Julia set which escape to infinity as well as a lower bound for the Hausdorff dimension of the radial limit set generated from a non-autonomous family of finite order meromorphic functions. To the best of the author's knowledge our results concerning the lower bound for the dimension of the set of escaping points is new for the given classes of meromorphic functions even in the autonomous case. 

In this article, we will primarily be concerned with non-autonomous dynamics stemming from perturbations of a single meromorphic function. In particular, given sequences $\seq{c_n}$ and $\seq{\lm_n}$ in $\CC$ and a transcendental meromorphic function $f$ of finite order $\rho$, we will consider additive and affine perturbations of $f$ defined by
\begin{align*}
f_n(z)=f(z)+c_n\spand \hat{f}_n(z)=\lm_n\cdot f(z)+c_n
\end{align*} 
for each $n\in\NN$. The non-autonomous additive and affine iterates are defined respectively for each $n\in\NN$ with 
\begin{align*}
F_+^n=f_n\circ\dots\circ f_1\spand F_A^n=\hat{f}_n\circ\dots\circ \hat{f}_1.
\end{align*}
In particular, by taking $c_n\equiv0$ and $\lm_n\equiv1$ for each $n\in\NN$, each of our results holds for ordinary, autonomous dynamical systems. 
If instead 
$$
c_n=c\neq 0 \spand \lm_n=\lm\neq 1
$$
are fixed for each $n\in\NN$ then our results apply to the perturbed autonomous systems given by 
$$
F_+=f_0+c \spand F_A=\lm\cdot f_0+c.
$$
Our results apply equally well to random dynamical systems if the perturbative parameters are chosen according to some probability distribution.
Additionally, by only taking the multiplicative perturbations $\lm_n\equiv 1$ and allowing the additive perturbations $c_n\neq 0$, we see that any statement concerning the function $F_A$ of non-autonomous affine perturbations also applies to the function $F_+$ of non-autonomous additive perturbations. This applies, in particular, to Theorem \ref{main thm: Mayer functions}. 

We will also give results concerning the eventual dimension and eventual hyperbolic dimension of a function. The eventual dimension of a transcendental meromorphic function is given by 
\begin{align*}
	\ED(f)=\lim_{R\to\infty}\HD\set{z\in\cJ(f): \absval{f^n(z)}>R, \;\forall n\geq 1},
\end{align*} 
and was first introduced by Rempe-Gillen and Stallard in \cite{rempe_hausdorff_2009} for entire functions, although the definition works for meromorphic functions as well. The concept of the eventual hyperbolic dimension of a transcendental function is a generalization of the notions of the eventual dimension and the hyperbolic dimension of Shishikura (see \cite{shishikura_boundary_1994}). The eventual hyperbolic dimension, which was first developed by De Zotti and Rempe-Gillen, is given by 
\begin{align*}
	\EHD(f)=\lim_{R\to\infty}\sup\set{\HD(X): X\sub\set{z:\absval{z}>R} \text{ is a hyperbolic set for } f }.
\end{align*}
In \cite{rempe_hyperbolic_2009}, Rempe-Gillen shows that the hyperbolic dimension of a transcendental function $f$ is equal to the Hausdorff dimension of the radial Julia set $\cJ_r(f)$, which is fully defined in Section~\ref{sec: non-aut mero dynamics}. Along the same vein as this result, we see that the eventual hyperbolic dimension and the Hausdorff dimension of the set 
\begin{align*}
	\cJ_r(f,R):=\set{z\in\cJ_r(f): \absval{f^n(z)}>R,\; \forall n\geq 1}
\end{align*} 
are similarly related.

\subsection{Structure of the Paper}
In Section~\ref{sec: Prelim} we recall some useful properties of meromorphic functions and certain notions from the study of non-autonomous dynamics as well as provide a statement of our main results. Section~\ref{Sec: NCIFS} concerns the necessary tools from the theory of non-autonomous iterated function systems. In Section~\ref{sec: Mayer} we will prove the first part of Theorem~\ref{main thm: Mayer functions}, and in Section~\ref{sec: KU} we will prove the first part of Theorems~\ref{main thm: escaping set KU} and \ref{main thm KU hyper}. In Section~\ref{sec ED} we will discuss the eventual and eventual hyperbolic dimensions of several well studied classes of functions as well as complete the proofs of our three main theorems. We will also make a connection with the eventual hyperbolic dimension and the Hausdorff dimension of the radial Julia set. Finally, in Section~\ref{sec: Examples} we will provide several examples of our main theorems.

\section{Preliminaries and Main Results}\label{sec: Prelim}
\subsection{Meromorphic Functions}
We refer the reader to the survey articles \cite{bergweiler_iteration_1993,kotus_fractal_2008} for a thorough treatment of the dynamics of meromorphic functions.   

In the sequel we will consider meromorphic functions of finite order $\rho=\rho(f)<\infty$. For $a\in\hat{\CC}$ we define the $a$-points to be the collection $f^{-1}(a)=\set{z_m(a):m\in\NN}$. Of particular interest will be Borel series of the form 
\begin{align*}
\Sg(t,a):=\sum_{m\in\NN}\absval{z_m(a)}^{-t}.
\end{align*} 
The exponent of convergence for the series is given by 
\begin{align*}
\rho_c(f,a):=\inf\set{t>0:\Sg(t,a)<\infty}. 
\end{align*}
A theorem of Borel shows that for all but at most two points $a\in\hat{\CC}$, we have that 
\begin{align}\label{Borel thm}
\rho_c(f,a)=\rho.
\end{align}
We say that a meromorphic function $f$ is of \textit{divergence type} if 
\begin{align*}
	\Sg(\rho,a)=\sum_{m\in\NN}\absval{z_m(a)}^{-\rho}=\infty.
\end{align*}
We let $\cP=\cP(f)$ denote the set of poles of $f$, and for each $a\in\cP$ we let $m(a)$ denote its multiplicity. If there are infinitely many poles and supposing that $\infty$ is not a Picard exceptional point, we see that the sum
\begin{align*}
	\sum_{a\in\cP}\absval{a}^{-t}
\end{align*}
converges for $t>\rho$ and diverges for $t<\rho$. If in addition there is some $M$ such that $m(a)\leq M$ for all $a\in\cP$, then there must also be some largest integer $M^*\leq M$ such that 
\begin{align}\label{def M*}
\sum_{a\in\cP\cap m^{-1}(M^*)}\absval{a}^{-t}
\end{align}
converges for $t>\rho$ and diverges for $t<\rho$.

By $\Sing(f^{-1})$ we denote the set of singular values, that is $z\in \Sing(f^{-1})$ if $z\in\CC$ and $z$ is a critical or asymptotic value of $f$. In the sequel we will consider functions from the Speiser class $\cS$ and the Eremenko-Lyubich class $\cB$ where
\begin{itemize}
	\item $f\in\cS$ if $\Sing(f^{-1})$ is finite, 	
	\item $f\in\cB$ if $\Sing(f^{-1})$ is bounded.
\end{itemize}
For more on these two classes of functions see \cite{eremenko_dynamical_1992}. In the sequel we will also require the use of the following result which is commonly known as Iversen's Theorem.
\begin{lemma}[Iversen's Theorem]\label{iversen lemma}
	Let $f$ be a transcendental meromorphic function such that $\infty$ is not an asymptotic value. Then, $f$ has infinitely many poles.
\end{lemma}
Since we will make significant use of them, we now present the celebrated distortion theorems due to Koebe.
\begin{theorem}[Koebe's $1/4$-Theorem] \label{KDT1/4}
If $z\in {\mathbb C}$, $r>0$ and $f:B(z,r)\to\CC$  is an arbitrary univalent analytic function, then
$$
f(B(z,r))\sub B(f(z),4^{-1}|f'(z)|r).
$$
\end{theorem}

\begin{theorem}[Koebe's Distortion Theorem]\label{KDT} There exists a function  $k:[0,1)\to [1,\infty)$ such that for
	any $z\in\CC$, $ r >0$, $ t\in [0,1)$ and any univalent
	analytic function $f:B(z,r)\to\CC$ we have that
	$$
	\sup\{|f'(w)|:w\in B(z,tr)\} \leq k(t) \inf\{|f'(w)|:w\in
	B(z,tr)\}.
	$$
\end{theorem}
In the sequel when we refer to Koebe's distortion constant $K$, we mean $K=k(1/2)$.
The following lemma is a straightforward consequence of the previous two
distortion theorems.
\begin{lemma}\label{lncp12.9.} Suppose that $D\sub\CC$ is an open
set, $z\in D$ and $f:D\to\CC$ is an analytic map which has
an analytic inverse $f_z^{-1}$ defined on $B(f(z),2R)$ for some
$R>0$. Then, for every $0\leq r\leq R$
\begin{align*}
B(z,K^{-1}r|f'(z)|^{-1})\sub f_z^{-1}(B(f(z),r))\sub
B(z,Kr|f'(z)|^{-1}).
\end{align*}
\end{lemma}

Throughout the article, for $R>0$ we let $B_R$ denote the set given by
\begin{align*}
	B_R=\set{z\in\CC:\absval{z}>R},
\end{align*}
and for $z\in\CC$ we let $B(z,r)$ and $\ol{B}(z,r)$ denote the respectively open and closed disks of radius $r$ centered at $z$. 
We will also use the symbols $\comp$ and $\lesssim$ to denote comparable values, by which we mean that $A\comp B$ if and only if there is some constant $C\geq 1$ such that $C^{-1}A\leq B\leq CA$, and $A\lesssim B$ if and only if there is $C\geq 1$ such that $A\leq CB$.

\subsection{Non-Autonomous Dynamics}\label{sec: non-aut mero dynamics}

Let $\sF=\set{f_\om}_{\om\in\Om}$ be a family of meromorphic functions. Given a sequence $\om:=(\om_n)_{n\in\NN}$ in $\Om$, we define the $n$\textsuperscript{th} iterate of the function $F_\om:\CC\to\hat{\CC}$ by 
\begin{align*}
	F^n_\om:=f_{\om_n}\circ\dots \circ f_{\om_1}:\hat{\CC}\to\hat{\CC}.
\end{align*}
If $\om$ is understood, we will write $F^n$ instead of $F_\om^n$ and $f_n$ instead of $f_{\om_n}$. 

We let $\cF(F_\om)$ be the set of points in $\CC$ such that the iterates $(F_\om^n)_{n=1}^\infty$ are defined and form a normal family on some neighborhood, and let $\jl(F_\om)=\cF(F_\om)^c$. Then, $\cF(F_\om)$ and $\jl(F_\om)$ are the non-autonomous Fatou and Julia sets associated with the fiber $\om$, respectively.
By 
\begin{align*}
I_\infty(F_\om)=\set{z\in\jl(F_\om):\limty{n}F_\om^n(z)=\infty}
\end{align*}
we denote the subset of the Julia set whose points escape to infinity under iteration of $F_\om$. The non-autonomous \textit{radial Julia set} associated with a given sequence $\om$, denoted by $\jl_r(F_\om)$, is the set of all points $z\in\jl(F_\om)$ such that $F_\om^n(z)$ is defined for all $n\in\NN$ and there is some $\dl>0$ such that for infinitely many $n\in\NN$, the disk $B(F^n_\om(z),\dl) $ can be pulled back univalently along the orbit of $z$.
\begin{lemma}\label{lem: deriv to infty implies Julia}
	If $\sF=\seq{f_n}$ is a sequence of meromorphic functions and $\xi\in\CC$ is a point such that there exists a sequence $\xi_k\to\xi$, $\xi_k\neq\xi$, and there is a subsequence $\seq[j]{n_j}$ such that
	\begin{align*}
	\limty{j}\absval{(F^{n_j})'(\xi_k)}=\infty,
	\end{align*}
	where $\seq[j]{F^{n_j}(\xi_k)}$ is bounded for all $k\geq 1$, then $\xi\in\jl(F)$.
\end{lemma}
\begin{proof}
	By way of contradiction, suppose $\xi\in\cF(F)$. Then, there is some sufficiently small neighborhood $U\ni\xi$ such that the iterates $F^n\rvert_U:U\to\hat{\CC}$ are defined, meromorphic, and form a normal family on $U$. Note that $U$ must not be a Baker domain, as we have that the sequence $\seq[j]{F^{n_j}(\xi_k)}$ remains bounded by assumption. So, without loss of generality, suppose that $F^{n_j}\rvert_U$ converges uniformly to some holomorphic function $g:U\to\hat{\CC}$ with $g(\xi)\in\CC$ and $g'(\xi)\neq \infty$. 
	By assumption, for all $k$ sufficiently large we have that 
	\begin{align*}
	g'(\xi_k)=\limty{j}\absval{(F^{n_j})'(\xi_k)}=\infty.
	\end{align*}	
	So, letting $k\to\infty$ gives that $g'(\xi)=\infty$, 	 
	which is a contradiction. Thus, we must have that $\xi\in\jl(F)$.
\end{proof}	
Given an initial function $f_0$ and sequences $c=(c_n)_{n=1}^\infty$ and $\lm=(\lm_n)_{n=1}^\infty$ of complex numbers, we can define the non-autonomous additive and affine perturbation functions $F_{+,c}$ and $F_{A,\lm,c}$, respectively, by first defining
\begin{align*}
f_n(z)=f(z)+c_n\spand \hat{f}_n(z)=\lm_n\cdot f(z)+c_n
\end{align*} 
for each $n\in\NN$, and then letting 
\begin{align*}
F_{+,c}^n=f_n\circ\dots\circ f_1\spand F_{A,\lm,c}^n=\hat{f}_n\circ\dots\circ \hat{f}_1.
\end{align*}
If in context the sequences $c$ and $\lm$ are clear, we simply write $F_+$ and $F_A$.
\subsection{Statement of Results}
The goal of this article is to show that for sufficiently small perturbative values, the dimensions of the escaping and radial sets of non-autonomous additive and affine functions, $F_+$ and $F_A$, have the same upper and lower bounds as the escaping and radial sets for the original unperturbed function $f$. In other words, we may use the dimension of the autonomous dynamical system to estimate the dimension of the non-autonomous system. We now present our three main results which concern the non-autonomous perturbations of two large and distinct classes of meromorphic functions of finite order. Our first result generalizes the results of \cite{mayer_size_2007}.
\begin{theorem}\label{main thm: Mayer functions}
	Let $f_0$ be a meromorphic function of finite order $\rho$ and suppose that the following hold. 
	\begin{enumerate}
		\item There exists a pole $b$ of $f_0$  such that  $b\not\in\ol{\Sing(f_0^{-1})}$. Let $q$ be the multiplicity of $b$.
		\item There are constants $s_0>0$, $Q>0$ and $\al>-1-1/q$ such that 
		\begin{align*}
		\absval{f_0'(z)}\leq Q\absval{z}^{\al} \qquad for \qquad z\in f_0^{-1}(U_0), \absval{z}\to\infty,
		\end{align*} 
		where $U_0=B(b,s_0)$. 
	\end{enumerate}
	Then, there exist $\dl>0$ and $\ep>0$ such that if $\seq{\lm_n}$ and $\seq{c_n}$ are sequences in $\CC$ such that 
	\begin{align*}
		\lm_n,\lm_n^{-1}\in B(1,\dl)\spand \absval{c_n}<\ep
	\end{align*}
	for each $n\in\NN$, then
	\begin{align*}
		\HD(\jl_{r}(F_A))\geq \frac{\rho}{\al+1+1/q}.
	\end{align*}
	Moreover, if $f_0$ has infinitely many such poles, then 
	\begin{align*}
		\EHD(f_0)\geq \frac{\rho}{\al+1+1/q}.	
	\end{align*}
	If, in addition, $f_0$ is of divergence type, then this last inequality is in fact strict. 
\end{theorem}
The following two theorems generalize the results of \cite{kotus_hausdorff_2008}, the first of which concerns the dimension of $I_\infty(F_+)$, while the second is concerned with the dimension of the radial Julia set $\cJ_r(F_A)$.

\begin{theorem}\label{main thm: escaping set KU}
	Let $f_0:\CC\to\hat{\CC}$ be a transcendental meromorphic function of finite order $\rho>0$ such that the following hold.
	\begin{enumerate}
		\item $\infty$ is not an asymptotic value of $f_0$.
		\item \label{main hyp 3} There exists a number $R^*>0$ and a co-finite subset $\cP^*\sub\cP$, i.e. $\cP\bs\cP^*$ is finite, such that 
		\begin{align*}
		\dist{\Sing(f^{-1})}{a}>R^*
		\end{align*} 
		for all $a\in\cP^*$.
		\item There exists $R^\dagger>0$ such that for distinct poles $a_1,a_2\in\cP$ we have 
		$$
			B(a_1,R^\dagger)\cap B(a_2,R^\dagger)=\emptyset.
		$$
		\item There exist $M\in\NN$ and $\bt\geq0$ such that for each pole $a\in\cP^*$ we have $1\leq m(a)\leq M$ and \label{main hyp 4}
		\begin{align*}
		\absval{f_0(z)}\comp\frac{\absval{a}^{-\bt}}{\absval{z-a}^{m(a)}} \quad and \quad 
		\absval{f_0'(z)}\comp\frac{m(a)\absval{a}^{-\bt}}{\absval{z-a}^{m(a)+1}}
		\end{align*}
		for $z\in B(a,R^\dagger)$.
	\end{enumerate}
	Then, there is $\ep>0$ such that if $\seq{c_n}$ is a sequence in $\CC$ with $\absval{c_n}<\ep$ for all $n\in\NN$, then 
	\begin{align*}
	\frac{\rho M^*}{\bt+M^*+1}&\leq
	\begin{cases}
	\HD(I_\infty(F_+)) \\
	\EHD(F_+) 
	\end{cases} 
	\leq \ED(F_+)\leq \frac{\rho M}{\bt+M+1},
	\end{align*}
	where $1\leq M^*\leq M$ is defined by \eqref{def M*}.
\end{theorem}
\begin{theorem}\label{main thm KU hyper}
	If $f_0$ has infinitely many poles and satisfies hypotheses \eqref{main hyp 3}-\eqref{main hyp 4} from the previous theorem, then for each $0\leq t<\frac{\rho M^*}{\bt +M^*+1}$ there exist $\ep_t,\dl_t>0$ such that if $c_t=\seq{c_n}$ and $\lm_t=\seq{\lm_n}$ are sequences in $\CC$ such that $\absval{c_n}<\ep_t$ and $\lm_n,\lm_n^{-1}\in B(1,\dl_t)$ for each $n\in\NN$, then
	\begin{align*}
	\HD(\jl_{r}(F_{A,\lm_t,c_t}))\geq\EHD(F_{A,\lm_t,c_t})\geq t.
	\end{align*}
	If, in addition, $f_0$ is of divergence type, then there exists $\ep,\dl>0$, no longer depending on $t$, such that if $c=\seq{c_n}$ and $\lm=\seq{\lm_n}$ are sequences in $\CC$ such that $\absval{c_n}<\ep$ and $\lm_n,\lm_n^{-1}\in B(1,\dl)$ for each $n\in\NN$, then we have
	\begin{align*}
	\HD(\jl_{r}(F_{A,\lm,c}))\geq\EHD(F_{A,\lm,c})\geq \frac{\rho M^*}{\bt +M^*+1}.
	\end{align*}
\end{theorem}
\begin{remark}\label{rem: main thm}
	Notice that Theorem \ref{main thm: escaping set KU} fully characterizes the dimension of the set $I_\infty$ for any meromorphic function which, in addition to satisfying the hypotheses of the theorem, has co-finitely many poles, all of which have the same multiplicity, in other words, $M^*=M$. This gives the following corollary.
\end{remark}
\begin{corollary}\label{cor: summary}
	Suppose $f_0$ satisfies the hypotheses of Theorem \ref{main thm: escaping set KU}. If we have that $M=M^*$, then there is $\ep>0$ such that if $\seq{c_n}$ is a sequence in $\CC$ with $\absval{c_n}<\ep$ for all $n\in\NN$, then 
	\begin{align*}
	\HD(I_\infty(F_+))=\EHD(F_+)=\ED(F_+)=\frac{\rho M}{\bt+M+1}\leq\HD(\cJ_r(F_+)).
	\end{align*}
	If, in addition, $f_0$ is of divergence type, then there are $\ep,\dl>0$ such that if $\seq{c_n}$, $\seq{\lm_n}$ are sequences in $\CC$ with $\absval{c_n}<\ep$ and $\lm_n,\lm_n^{-1}\in B(1,\dl)$ for all $n\in\NN$, then 
	\begin{align*}
	\HD(I_\infty(F_+))=\EHD(F_+)=\ED(F_+)=\frac{\rho M}{\bt+M+1}\leq \EHD(F_A)\leq\HD(\cJ_r(F_A)).
	\end{align*}
\end{corollary}
\begin{remark}
	An alternate interpretation of Corollary~\ref{cor: summary} is that the Hausdorff dimension of the set $I_\infty(f_0)$ is stable under sufficiently small additive perturbations. In particular we see that for a function $f_0$ meeting the hypotheses of Corollary~\ref{cor: summary} there is $\ep>0$ such that for any $0<c<\ep$, we have 
	\begin{align*}
		\HD(I_\infty(f_0))=\HD(I_\infty(f_0+c))=\frac{\rho M}{\bt+M+1}.
	\end{align*} 
\end{remark}
\section{Non-Autonomous Conformal Iterated Function Systems}\label{Sec: NCIFS}
The main technique used throughout this article will be to build a non-autonomous iterated function system whose limit set sits comfortably within the Julia set. We now recall some properties of non-autonomous IFSs. 

\begin{definition}
	For each $n\geq 0$, we let $X_n\sub\RR^d$ be a compact, connected set which is regularly closed, i.e. 
	$$
		X_n=\ol{\intr{X_n}}.
	$$ 
	A \textit{non-autonomous conformal iterated function system} (NCIFS) $\Phi$ on the sequence $(X_n)_{n=0}^\infty$ is a sequence $\Phi^{(1)}, \Phi^{(2)}, \dots$ where for each $n\geq 1$, 
	$$
		\Phi^{(n)}=\set{\phi_i^{(n)}:X_n\to X_{n-1}\;: i\in I^{(n)}}
	$$ 
	is a countable collection of contractions such that the following seven conditions are satisfied.
	\begin{itemize}
		\item (Open Set Condition) For all $j\in\NN$ and all $a\neq b\in I^{(n)}$ we have 
		\begin{align*}
		\phi_a^{(n)}(\intr{X_n})\cap\phi_b^{(n)}(\intr{X_n})=\emptyset.
		\end{align*}
		
		\item (Conformality) There exists an open connected set $W_n\bus X_n$ (independent of $i\in I^{(n)}$) such that each map $\phi_i^{(n)}$ extends to a $C^1$ conformal diffeomorphism of $W_n$ into $W_{n-1}$. 
		
		\item (Bounded Distortion) There exists a constant $K\geq 1$ such that for any $k\leq \ell$ and any word $\om_k\om_{k+1}\dots\om_\ell$ with $\om_j\in I^{(j)}$ for each $k\leq j\leq \ell$, the map $\phi=\phi_{\om_k}^{(k)}\circ\dots\circ\phi_{\om_\ell}^{(\ell)}$ satisfies
		\begin{align*}
		\absval{D\phi(x)}\leq K\absval{D\phi(y)}
		\end{align*}
		for all $x,y\in W_n$.
		
		\item (Uniform Contraction) There is a constant $\beta<1$ such that 
		\begin{align*}
		\absval{D\phi(x)}<\beta^m
		\end{align*}
		for all sufficiently large $m\in\NN$, all $x\in X_n$, and all maps $\phi=\phi_{\om_j}^{(j)}\circ\dots\circ\phi_{\om_{j+m}}^{(j+m)}$ where $j\geq 1$ and $\om_k\in I^{(k)}$ for each $k$. 	
		
		\item (Geometry Condition): 	There exists $N\in\NN$ such that for all $n\in\NN$ there exist $\Gm_1^{(n)},\dots\Gm_N^{(n)}\sub W_{n}$ such that each of the $\Gm_j^{(n)}$ are convex and 
		\begin{align*}
		X_{n}\sub\union_{j=1}^N\Gm_j^{(n)}.
		\end{align*}
		We also suppose there exists $\vta>0$ such that for each $x\in X_n$ we have that
		\begin{align*}
			B(x,\vta\cdot\diam(X_{n}))\sub W_{n}.
		\end{align*}		
		
		\item (Uniform Cone Condition): There exist $\al,\gm>0$ with $\gm<\frac{\pi}{2}$ such that for every $n\in\NN$ and every $x\in X_{n}$ there is an open cone 
		\begin{align*}
			Con(x,u_x,\gm,\al\cdot\diam(X_{n}))\sub\intr{X_{n}}
		\end{align*} 
		with vertex $x$, direction vector $u_x$, central angle of measure $\gm$, and altitude $\al\cdot\diam(X_{n})$ comparable to $\diam(X_{n})$.  
		
		\item (Diameter Condition): For each $n\in\NN$ we have
		\begin{align*}
			\limty{n}\frac{1}{n}\log \diam(X_{n})=0 \spand \limty{n}\frac{1}{n}\sup_{k\geq 0}\log\frac{\diam(X_{k+n})}{\diam(X_{k})}=0.
		\end{align*}	
	\end{itemize}
\end{definition}
\begin{definition}
	A NCIFS $\Phi$ is called \textit{stationary} if the sequence of sets $(X_n)_{n=0}^\infty$ is constant, i.e. if $X_n=X_m$ for all $n,m\geq 0$. To emphasize when a particular NCIFS is not stationary, we will call that system \textit{non-stationary}. 
	The system $\Phi$ is called \textit{finite} if the collections $\Phi^{(n)}$ are finite for each $n$, and \textit{infinite} otherwise. $\Phi$ is said to be \textit{uniformly finite} if there is a constant $M>0$ such that $\#I^{(n)}<M$ for each $n\in\NN$. 
\end{definition}
\begin{remark}
	Notice that if each of the spaces $X_n$ is convex, then the Uniform Cone Condition and the Geometry Condition hold. Furthermore, Koebe's Distortion Theorem implies that the Bounded Distortion Property holds for dimension $d=2$. If the system $\Phi$ is in fact a stationary NCIFS, then the Uniform Cone Condition, Geometry Condition, and Diameter Condition are automatically satisfied.  
\end{remark}
The limit set of a NCIFS $\Phi$ is defined as 
\begin{align*}
		J_\Phi:=\intersect_{n=1}^\infty\union_{\om\in I^n}\phi_\om(X_n)\sub X_0,
\end{align*}
where 
$$
	\phi_\om:=\phi_{\om_1}^{(1)}\circ\dots\circ\phi_{\om_n}^{(n)}.
$$
For each $n\in\NN$ and $0\leq t\leq d$ we define the potential functions
\begin{align*}
	Z_n(t)=\sum_{\om\in I^n}\norm{(\phi_\om)'}^t \quad \text{ and }\quad Z_{(n)}(t)=\sum_{i\in I^{(n)}}\norm{(\phi_i^{(n)})'}^t,
\end{align*}
where we take $\norm{\spot}$ to denote the sup norm.
Bounded distortion implies that
\begin{align}\label{eqn: Zn ineq}
	Z_n(t)\geq K^{-nt}Z_{(1)}(t)\cdots Z_{(n)}(t).
\end{align}
The lower pressure function can then be defined as 
\begin{align*}
	\ul{P}(t)=\liminfty{n}\frac{1}{n}\log Z_n(t).
\end{align*}
We say that \textit{Bowen's formula holds} for the system $\Phi$ if the Hausdorff dimension of the limit set coincides with the Bowen dimension of the limit set, that is if 
\begin{align*}
	\HD(J_\Phi)=B_\Phi,
\end{align*}
where the Bowen dimension $B_\Phi$ is given by
\begin{align*}
	B_\Phi&:=\sup\set{t\geq 0:\ul{P}(t)\geq 0}=\inf\set{t\geq 0: \ul{P}(t) \leq 0}=\sup\set{t\geq 0:Z_n(t)\to\infty}.	
\end{align*}
\begin{definition}
	We say that a NCIFS $\Phi$ is \textit{subexponentially bounded} if 
	\begin{align*}
		\lim_{n\to\infty}\frac{1}{n}\log\#I^{(n)}=0.
	\end{align*}
\end{definition}
In the sequel we will make use of the following theorem.
\begin{theorem}\label{BF for NCIFS}
	If $\Phi$ is a finite, subexponentially bounded NCIFS, then Bowen's formula holds.
\end{theorem}
See \cite{rempe-gillen_non-autonomous_2016} for a proof in the stationary setting, and \cite{atnip_nonautonomous_2017} for a proof in the non-stationary setting. 

\section{Radial Set for Affine Perturbations}\label{sec: Mayer}
Throughout this section we consider a general meromorphic function $f_0$ of finite order $\rho$ like those considered by Mayer in \cite{mayer_size_2007}. We suppose that the following hold. 
\begin{enumerate}
	\item[\mylabel{(M1)}{(M1)}] There exists a pole $b$ of $f_0$ such that $b\not\in\ol{\Sing(f_0^{-1})}$. Let $b$ be such a pole and let $q=m(b)$ be the multiplicity of $b$.
	\item[\mylabel{(M2)}{(M2)}] There are constants $s_0>0$, $Q>0$ and $\al>-1-1/q$ such that 
	\begin{align}\label{eqn: M derivative condition}
	\absval{f_0'(z)}\leq Q\absval{z}^{\al} \quad \text{for} \quad z\in f_0^{-1}(U_0),\; \absval{z}\to\infty,
	\end{align} 
	where $U_0=B(b,s_0)$. 
\end{enumerate}

The proof of the main theorem of this section will rely on our ability to construct a finite, stationary NCIFS, whose limit set is contained within the non-autonomous Julia set $\jl(F_A)$, for which we can find a suitable lower bound for its Hausdorff dimension. 

\begin{theorem}\label{thm: Mayer functions}
	Let $f_0$ be a meromorphic function of finite order $\rho$ satisfying conditions \ref{(M1)} and \ref{(M2)} above with pole $b$ and neighborhood $U_0=B(b,s_0)$. Then, there exist $\ep,\dl>0$ such that if $\seq{\lm_n}$ and $\seq{c_n}$ are sequences in $\CC$ such that 
	\begin{align*}
		\lm_n,\lm_n^{-1}\in B(1,\dl)\spand \absval{c_n}<\ep
	\end{align*}
	for each $n\in\NN$, then 
	$$
		\HD(\jl_{r}(F_A))\geq \frac{\rho}{\al+1+1/q}.
	$$
\end{theorem}
\begin{proof}
	
	Suppose $b$ is a pole of $f_0$ with multiplicity $q$.  Note that near the pole $b$, $f_0$ is of the form 
	$$
		f_0(z)=\frac{g_0(z)}{(z-b)^q}
	$$
	where $g_0$ is a function which is analytic on a neighborhood of $b$ with $g_0(b)\neq 0$. Without loss of generality we may assume that $s_0$ is sufficiently small such that the following hold:
	\begin{enumerate}
		\item[(i)] No singular values of $f_0$ belongs to $U^*:=B(b,2s_0)$, i.e. $\Sing(f_0^{-1})\cap U^*=\emptyset$,
		\item[(ii)] For each $w\in U^*\bs\set{b}$ we have 
		\begin{align}\label{other deriv cond}
		\absval{f_0'(w)} \comp \frac{1}{\absval{w-b}^{q+1}}\comp\absval{f_0(w)}^{1+1/q}.
		\end{align}
	\end{enumerate} 
	Now let 
	\begin{align*}
	V:=f_0(U_0\bs\set{b}),
	\end{align*} 
	which is a nonempty punctured neighborhood of $\infty$. Choose $R_0>0$ sufficiently large such that $B_{R_0}\sub V$. Recall that the $b$-points for $f_0$ is the collection of points 
	\begin{align*}
		\set{z_m^{(0)}:m\in\NN}=f_0^{-1}(b).	
	\end{align*}
	Then, $z_m^{(0)}\to\infty$ as $m\to\infty$. For each $z_m^{(0)}\in f_0^{-1}(b)\cap V$ let $\phi_m^{(0)}:U_0\to\CC$ denote the holomorphic inverse branch of $f_0$ such that $\phi_m^{(0)}(b)=z_m^{(0)}$. In Claim 3.1 of \cite{mayer_size_2007}, Mayer shows that there exists $M_0\in\NN$ such that for all $m\geq M_0$, 
	\begin{align*}
		\ol{\phi_m^{(0)}(U_0)}\sub B_{R_0}\sub V.
	\end{align*}
	As $f_0:U_0\to V$ is an unbranched covering of $V$, for any $\Om\sub B_{R_0}\bs\set{\infty}$ open and simply connected, we define $\psi_\Om$ to be an inverse branch of $f_0:U_0\bs\set{b}\to V$ defined on $\Om$ such that 
	\begin{align*}
		\psi_\Om(\Om)\sub U_0\bs\set{b}. 
	\end{align*} 
	In particular, for each $m\geq M_0$, we define $\psi_m^{(0)}$ to be an inverse branch of $f_0$ defined on $\phi_m^{(0)}(U_0)$. Mayer then shows that the infinite autonomous iterated function system given by 
	\begin{align*}
	\Phi_0=\set{\gm_m^{(0)}:\ol{U}_0\to \ol{U}_0 : m\geq M_0},
	\end{align*}
	where $\gm_m^{(0)}:=\psi_m^{(0)}\circ\phi_m^{(0)}$, is such that the limit set $J_{\Phi_0}$ of $\Phi_0$ is contained within the Julia set $\jl(f_0)$. Mayer is then able to estimate that 
	\begin{align*}
	\HD(J_{\Phi_0})\geq\frac{\rho}{\al+1+1/q}
	\end{align*} 
	by showing that 
	\begin{align*}
	\sum_{m\geq M_0}\absval{(\gm_m^{(0)})'(b)}\geq\sum_{m\geq M_0}\absval{z_m^{(0)}}^{-t(\al+1+1/q)}.
	\end{align*} 
	In view of \eqref{Borel thm}, we see that $\frac{\rho}{\al+1+1/q}$ is the critical exponent of the full series 
	\begin{align*}
	\sum_{m\geq 1}\absval{z_m^{(0)}}^{-t(\al+1+1/q)}.
	\end{align*}
	Thus, he obtains the result by applying Theorem 3.15 of \cite{mauldin_dimensions_1996}, which says that the Hausdorff dimension of the limit set of an infinite (autonomous) conformal IFS $S=\set{\gm_j:X\to X}_{j\in\NN}$ on a set $X$ is at least as large as 
	\begin{align}\label{theta}
	\theta:=\inf\set{t>0:\sum_{j\in\NN}\absval{\gm_j'(x)}^t<\infty, x\in X}.
	\end{align}  
	Now we build off Mayer's construction to prove our result. We begin by letting 
	\begin{align}\label{eqn:affine M r choice}
	0<s_1<\frac{s_0}{16K^2},
	\end{align} 
	where $K$ comes from Koebe's Distortion Theorem. Let $U_1=B(b,s_1)$, and let $R_1\geq R_0$ such that 
	\begin{align}\label{Psi_Om contained in U_1}
		\psi_\Om(\Om)\sub U_1
	\end{align}
	for all $\Om\sub B_{R_1}$ open and simply connected. Let $R_2\geq 3R_1$, and we choose $\ep, \dl>0$ such that the following hold:
	\begin{itemize}
		\item $\dl<\min\set{\frac{s_0}{8\absval{b}},\frac{s_0}{8},\frac{1}{2}}$,	
		\item $(1+\dl)\left[r+\dl(1+\absval{b})\right]<s_0$,
		\item $\ep<\min\set{\frac{s_1}{2},\dl, R_1}$.
	\end{itemize}
	To see that such a $\dl>0$ does in fact exist, we note that $(1+\dl)\left[s_1+\dl(1+\absval{b})\right]<s_0$ if
	\begin{align*}
		0<\dl<\frac{-1+\sqrt{1+4\frac{s_0-s_1}{1+s_1+\absval{b}}}}{2}.
	\end{align*} 	
	Since $z_m^{(0)}\to\infty$, we take $M_1\geq M_0$ sufficiently large such that 
	\begin{align}\label{eqn:affine M choice of M1}
		\phi_m^{(0)}(U_0)\sub B_{R_2}
	\end{align} 
	for all $m\geq M_1$.	
	Then, by our choice of $\ep, \dl$ we have that 
	\begin{align*}
		\frac{R_2-\ep}{1+\dl}\geq R_1,
	\end{align*}	
	since 
	\begin{align*}
		(1+\dl)R_1+\ep \leq2R_1+R_1=3R_1\leq R_2.
	\end{align*}
	For each $n\in\NN$, we suppose that $\absval{c_n}<\ep$ and that both $\lm_n,\lm_n^{-1}\in B(1,\dl)$. Furthermore, for each $w\in \CC$, define 
	\begin{align*}
		f_n(w)=\lm_nf_0(w)+c_n.
	\end{align*}
	We now claim that for each $n\in \NN$ we have that 
	\begin{align}\label{perturbed U_1 in U_0}
		b\in \frac{U_1-c_n}{\lm_n}\sub U_0,
	\end{align}
	where 
	$$
		\frac{U_1-c_n}{\lm_n}:=\set{\frac{w-c_n}{\lm_n}:w\in U_1}.
	$$ 
	To see this we simply calculate 
	\begin{align*}
		\absval{\frac{w-c_n}{\lm_n}-b}&\leq \absval{\frac{w-c_n}{\lm_n}-\frac{b-c_n}{\lm_n}}+\absval{\frac{b-c_n}{\lm_n}-b}\\
		&=\absval{\lm_n^{-1}}\absval{w-b}+\absval{\lm_n^{-1}}\absval{b-c_n-\lm_nb}\\
		&\leq(1+\dl)\left[s_1+\absval{c_n}+\absval{1-\lm_n}\absval{b}\right]\\
		&\leq (1+\dl)\left[s_1+\dl(1+\absval{b})\right]<s_0.
	\end{align*}
	For each $n\in\NN$ and $m\geq M_1$ we define the inverse branch $\phi_m^{(n)}$ of $f_n$ on $U_1$ by
	\begin{align*}
		\phi_m^{(n)}(w)=\phi_m^{(0)}\left(\frac{w-c_n}{\lm_n}\right),
	\end{align*}
	and hence, by \eqref{eqn:affine M choice of M1} and \eqref{perturbed U_1 in U_0}, we have that 
	\begin{align}\label{phi_m^n(U_1) sub BR2}
		\phi_m^{(n)}(U_1)\sub\phi_m^{(0)}(U_0)\sub B_{R_2}
	\end{align}
	for all $m\geq M_1$ and $n\in\NN$. We also note that condition \eqref{eqn: M derivative condition} above ensures that 
	\begin{align*}
		\absval{(\phi_m^{(n)})'(w)}\geq K^{-1}\absval{\phi_m^{(n)}(w)}^{-\al},
	\end{align*} 
	for each $w\in U_1$. 	
	
	Now, given \eqref{Psi_Om contained in U_1} and \eqref{eqn:affine M choice of M1} we have that 
	\begin{align}\label{Om_m^n in B_R_1}
		\Om_m^{(n)}:=\frac{\phi_m^{(0)}(U_0)-c_{n}}{\lm_{n}}\sub B_{R_1}
	\end{align}
	for all $n\in\NN$ and all $m\geq M_1$. Define 
	\begin{align*}
		\psi_m^{(n)}:=\psi_{\Om_m^{(n)}}:\Om_m^{(n)}\longrightarrow U_0
	\end{align*}
	to be an inverse branch of $f_0:U_0\bs\set{b}\to V$ on $\Om_m^{(n)}$. 
	Equivalently, we see that $\psi_m^{(n)}$ is an inverse branch of $f_n:U_1\bs\set{b}\to V$ defined on $\phi_m^{(0)}(U_0)$. 
	In light of \eqref{Psi_Om contained in U_1}, \eqref{eqn:affine M choice of M1}, and \eqref{Om_m^n in B_R_1} we see that 
	\begin{align}\label{psi_m^n of Om in U_1}
		\psi_m^{(n)}(\Om_m^{(n)})\sub U_1\sub U_0
	\end{align}
	for all $n\in\NN$ and all $m\geq M_1$.   
	In view of \eqref{Om_m^n in B_R_1} and \eqref{psi_m^n of Om in U_1} we see that
	\begin{align*}
		\psi_m^{(n-1)}(\phi_m^{(n)}(U_1))\sub U_1
	\end{align*}
	for all $n\in\NN$ and all $m\geq M_1$.
	Now, define   
	\begin{align*}
		\gm_m^{(n)}:=\psi_m^{(2n-1)}\circ\phi_m^{(2n)}:\ol{U}_1\to \ol{U}_1	
	\end{align*}
	for all $n\in\NN$ and $m\geq M_1$.

	Fix $t<\frac{\rho}{\al+1+1/q}\leq2$. Then, 
	\begin{align}\label{Mayer sum diverge}
		\sum_{m\geq M_1}\absval{z_m^{(0)}}^{-t(\al+1+1/q)}=\infty,
	\end{align}
	and thus there is some $N_t\in\NN$, depending on $t$, such that 
	\begin{align}\label{choice of M}
		\sum_{m=M_1}^{M_1+N_t}\absval{z_m^{(0)}}^{-t(\al+1+1/q)}\geq 2^{1+2(4+2/q)}K^2Q^2L                                    
	\end{align} 
	where $L\geq 1$ is the comparability constant coming from \eqref{other deriv cond}.

	Letting $I^{(n)}=\set{m\in\NN:M_1\leq m\leq M_1+N_t}$, for each $n\in\NN$ we set 
	\begin{align*}
		\Phi^{(n)}:=\set{\gm_m^{(n)}:\ol{U}_1\to \ol{U}_1	: m\in I^{(n)}}, 
	\end{align*}
	and let 
	\begin{align*}
		\Phi:=\seq{\Phi^{(n)}}.
	\end{align*}
	Note that while the alphabets $I^{(n)}$ do not depend upon $n$, the collection $\Phi^{(n)}$ does depend on $n$.
	Since the images of the inverse branches are disjoint, the open set condition is satisfied, and as $\#I^{(n)}=N_t+1$ for each $n\in\NN$, we have that $\Phi$ is a uniformly finite, stationary NCIFS. Thus, Bowen's formula holds, i.e.
	\begin{align*}
		\HD(J_\Phi)=B_\Phi.
	\end{align*}  
	Now, to see that $J_\Phi\sub\jl(F_A)$ as desired, suppose $z\in J_\Phi$. Since $\absval{\gm_\om'(z)}\to 0$ as $\absval{\om}=n\to\infty$, where $\absval{\om}$ denotes the length of the word $\om$, then we see $\absval{(F^{2n}_A)'(z)}\to\infty$ as $n\to\infty$. Thus, applying Lemma \ref{lem: deriv to infty implies Julia}, we see that the limit set $J_\Phi$ is contained in the Julia set $\jl(F_A)$. Furthermore, by construction, we have that $J_\Phi\sub\jl_{r}(F_A)$.

	Now we estimate a lower bound of the Hausdorff dimension of the limit set of $\Phi$ analogous to $\theta$ given in \eqref{theta} for the autonomous setting (comp. \cite{rempe-gillen_non-autonomous_2016}, \cite{atnip_nonautonomous_2017}).
	
	Let $w_m^{(n)}=\lm_{2n}b+c_{2n}$ for each $n\in\NN$ and $M_1\leq m\leq M_1+N_t$. Notice that 
	\begin{align*}
		\phi_m^{(2n)}(w_m^{(n)})=z_m^{(0)}.
	\end{align*}
	Then, 
	\begin{align*}
		Z_{(n)}(t)&\geq \sum_{m=M_1}^{M_1+N_t}\absval{(\gm_m^{(n)})'(w_m^{(n)})}^t\\
		&=\sum_{m=M_1}^{M_1+N_t}\absval{(\psi_m^{(2n-1)})'(z_m^{(0)})}^{t}\absval{(\phi_m^{(2n)})'(w_m^{(n)})}^{t}\\
		&=\sum_{m=M_1}^{M_1+N_t}\absval{f_{2n-1}'(\psi_m^{(2n-1)}(z_m^{(0)}))}^{-t}\absval{f_{2n}'(z_m^{(0)})}^{-t}\\
		&=\sum_{m=M_1}^{M_1+N_t}\absval{\lm_{2n-1}}^{-t}\absval{f_{0}'(\psi_m^{(2n-1)}(z_m^{(0)}))}^{-t}\absval{\lm_{2n}}^{-t}\absval{f_{0}'(z_m^{(0)})}^{-t}\\
		&\geq (1-\dl)^{2t}\sum_{m=M_1}^{M_1+N_t}\absval{f_{0}'(\psi_m^{(2n-1)}(z_m^{(0)}))}^{-t}\absval{f_{0}'(z_m^{(0)})}^{-t}.
	\end{align*}
	As $b\neq\psi_m^{(2n-1)}(z_m^{(0)})\in U_1$ and $w_m^{(n)}\in U_1$, applying \eqref{eqn: M derivative condition}, \eqref{other deriv cond}, and \eqref{choice of M} we see 
	\begin{align*}
		Z_{(n)}(t)&\geq Q^{-t}(1-\dl)^{2t}\sum_{m=M_1}^{M_1+N_t}\absval{f_{0}'(\psi_m^{(2n-1)}(z_m^{(0)}))}^{-t}\absval{z_m^{(0)}}^{-t\al}\\
		&\geq Q^{-t}L^{-1}(1-\dl)^{2t}\sum_{m=M_1}^{M_1+N_t}\absval{f_{0}(\psi_m^{(2n-1)}(z_m^{(0)}))}^{-t(1+1/q)}\absval{z_m^{(0)}}^{-t\al}\\	
		&= Q^{-t}L^{-1}(1-\dl)^{2t}\sum_{m=M_1}^{M_1+N_t}\absval{f_{0}\left(\psi_m^{(0)}\left(\frac{z_m^{(0)}-c_{2n-1}}{\lm_{2n-1}}\right)\right)}^{-t(1+1/q)}\absval{z_m^{(0)}}^{-t\al}\\	
		&= Q^{-t}L^{-1}(1-\dl)^{2t}\absval{\lm_{2n-1}}^{t(1+1/q)}\sum_{m=M_1}^{M_1+N_t}\absval{z_m^{(0)}-c_{2n-1}}^{-t(1+1/q)}\absval{z_m^{(0)}}^{-t\al}\\	
		&\geq Q^{-t}L^{-1}(1-\dl)^{t(3+1/q)}\sum_{m=M_1}^{M_1+N_t}\absval{z_m^{(0)}-c_{2n-1}}^{-t(1+1/q)}\absval{z_m^{(0)}}^{-t\al}.
	\end{align*}
	Given that $t<2$, $\dl<\frac{1}{2}$, and that $\absval{z_m^{(0)}-c_{2n-1}}\leq2\absval{z_m^{(0)}}$ for all $n\in\NN$ and $M_1\leq m\leq M_1+N_t$ (if this were not the case we could otherwise take $M_1$ sufficiently large), we have
	\begin{align*}
	Z_{(n)}(t)&\geq Q^{-t}L^{-1}2^{-t(1+1/q)}(1-\dl)^{t(3+1/q)}\sum_{m=M_1}^{M_1+N_t}\absval{z_m^{(0)}}^{-t(1+1/q)}\absval{z_m^{(0)}}^{-t\al}\\
	&\geq Q^{-t}L^{-1}2^{-t(4+2/q)}\sum_{m=M_1}^{M_1+N_t}\absval{z_m^{(0)}}^{-t(\al+1+1/q)}\\
	&\geq Q^{-2}L^{-1}2^{-2(4+2/q)}\sum_{m=M_1}^{M_1+N_t}\absval{z_m^{(0)}}^{-t(\al+1+1/q)}\geq 2K^2.
	\end{align*}	
	Thus, in light of \eqref{eqn: Zn ineq}, we see that
	\begin{align*}
		Z_n(t)\geq 2^n,
	\end{align*}
	which in turn implies that $\ul{P}(t)>0$, and hence $\HD(J_\Phi)\geq t$. As this holds for each $t<\frac{\rho}{\al+1+1/q}$, we reach the conclusion that 
	\begin{align*}
		\HD(\cJ_r(F_A))\geq\frac{\rho}{\al+1+1/q},
	\end{align*} 
	which finishes the proof. 
\end{proof}
The following result of Mayer \cite{mayer_size_2007} follows in part from our Theorem \ref{thm: Mayer functions}. 
\begin{corollary}
	Let $f$ be a meromorphic function of finite order $\rho$ which satisfies the hypotheses of the previous theorem, including the constants $\al$ and $q$. Then, 
	\begin{align*}
		\HD(\cJ_r(f))\geq \frac{\rho}{\al+1+1/q}.
	\end{align*}
	If, in addition, $f$ is of divergence type, then the inequality becomes strict. 
\end{corollary}
\begin{remark}
	We are unable to prove a corresponding statement concerning functions of divergence type as the theory of non-autonomous iterated function systems is not as developed as the theory of infinite iterated function systems with respect to the number $\ta$, compare \cite{atnip_nonautonomous_2017} and \cite{mauldin_graph_2003}.
\end{remark}
\section{Escaping Set for Additive Perturbations}\label{sec: KU}
We now examine the class of functions investigated by Kotus and Urba\'nski in \cite{kotus_hausdorff_2008}.
Let $\cP=\cP(f)=f^{-1}(\infty)$ denote the set of all poles of the function $f$. Let $m$ be the function on the set of poles $\cP$ which assigns to each pole $a$ its multiplicity $m(a)$. 
In this section we will consider a transcendental meromorphic function, $f_0:\CC\to\hat{\CC}$, of finite order $\rho>0$ such that the following hold.
\begin{enumerate}
	\item[\mylabel{(KU1)}{(KU1)}] $\infty$ is not an asymptotic value of $f_0$.
	\item[\mylabel{(KU2)}{(KU2)}] There exists a co-finite subset $\cP^*\sub\cP$, which means precisely that $\cP\bs\cP^*$ is finite, and there exists $R^*>0$ such that 
	$$
		\dist{\Sing(f^{-1})}{a}>2R^*
	$$
	for all $a\in\cP^*$.
	\item[\mylabel{(KU3)}{(KU3)}] There exists $R^\dagger>0$ such that for distinct poles $a_1,a_2\in\cP$ we have 
	$$
	B(a_1,R^\dagger)\cap B(a_2,R^\dagger)=\emptyset.
	$$
	\item[\mylabel{(KU4)}{(KU4)}] There exist $M\in\NN$ and $\bt\geq 0$ such that for each $a\in\cP^*$
	\begin{align*}
		\absval{f_0(z)}\comp\frac{\absval{a}^{-\bt}}{\absval{z-a}^{m(a)}} \quad and \quad 
		\absval{f_0'(z)}\comp\frac{m(a)\absval{a}^{-\bt}}{\absval{z-a}^{m(a)+1}}
	\end{align*}
	for $z\in B(a,R^\dagger)$, where $m(a)\in\NN$ with $1\leq m(a)\leq M$.
\end{enumerate}
Note that Lemma \ref{iversen lemma} implies that $f_0$ has infinitely many poles. As $m:\cP\to\NN$ takes on only finitely many values, there is $M\in\NN$ such that $m(a)\leq M$ for each $a\in\cP$ and there is a largest integer $M^*\leq M$ such that the sum 
\begin{align*}
\sum_{a\in m^{-1}(M^*)}(1+\absval{a})^{-t}
\end{align*}
is finite for $t>\rho$ and infinite for $t<\rho$.

\begin{theorem}\label{UB of dimension of escaping set}
	If $f_0$ satisfies the above conditions \ref{(KU1)}-\ref{(KU4)}, then there is $\ep>0$ such that if $\seq{c_n}$ is a sequence in $\CC$ with $\absval{c_n}<\ep$ for all $n\in\NN$, then 
	\begin{align*}
		\HD(I_\infty(F_+))\leq\frac{\rho M}{\bt+M+1}.
	\end{align*}
\end{theorem}
\begin{remark}\label{rem: escapers stay close to poles}
	The idea behind the proof relies on fact that $\infty$ is not an asymptotic value of $f_0$, nor $f_n$ for any $n$. This means that points which escape to infinity under iterates of $F_{+}$ must remain close to poles.  
\end{remark}
\begin{proof}
	Let $S^*=\min\set{2R^*,R^\dagger}$. Then, 
\begin{align}
	\dist{\Sing(f_0^{-1})}{\cP^*}>2S^*. \label{eqn: KU2}
\end{align} 
Let $0<S<S^*/2$ and choose $0<\ep<S^*-2S$. Then, the disks $B(a,S^*)$ are mutually disjoint as for $a\in\cP^*$. Taking $R_0$ sufficiently large with $R_0\geq \max\set{2R^*, R^\dagger}$, we have that 
\begin{align*}
	B_{R_0}\sub f_0(B(a,S^*))	
\end{align*}
 for each $a\in\cP^*$ since
\begin{align*}
	\absval{a}^{-\bt}(S^*)^{-m(a)}\lesssim\absval{a}^{-\bt}\lesssim 1.
\end{align*}
For $R>0$ denote 
\begin{align*}
	\cP_R:=\cP\cap B_R.
\end{align*}
Let $R_1\geq R_0$ sufficiently large such that 
\begin{align*}
	B_{R_0}\sub f_0(B(a,S^*))\spand \dist{\Sing(f_0^{-1})}{a}>2R^*
\end{align*}
for all $a\in\cP_{R_1}$, which must exist since $\cP\bs\cP^*$ is finite. For each $a\in\cP$ and $R>0$ we let $B_a(R)$ denote the connected component of $f_0^{-1}(B_R)$ which contains $a$. Then, for $R\geq R_0$ and $a\in\cP_{R_1}$ we have 
\begin{align}\label{B_a(R)sub}
	B_a(R)\sub B(a,S^*).
\end{align}
Now, hypothesis \ref{(KU4)} also implies that there is a constant $L\geq 1$ such that for all $a\in\cP$ and all $R\geq2R^*$ we have 
\begin{align}
\diam(B_a(R))\leq LR^{-1/m(a)}\absval{a}^{-\bt/m(a)}  \label{eqn: KU5}.
\end{align}
Choose $R_2\geq R_1$ sufficiently large such that for all $R\geq R_2$ we have 
\begin{align}\label{diam ineq with S}
	\diam(B_a(R))\leq LR^{-1/m(a)}\absval{a}^{-\bt/m(a)}
	\leq LS^{-1/m(a)}\absval{a}^{-\bt/m(a)}.
\end{align}
If 
\begin{align*}
	U\sub \left(B_{R_2}\bs\set{\infty}\right)\cap\left(\union_{a\in\cP}B(a,2R^*)\right)	
\end{align*}
is open and simply-connected, then all holomorphic inverse branches $f_{0,a,U,j}^{-1}$ of $f_0$, which take $U$ into $B(a,R^*)$, are all well defined for $1\leq j\leq m(a)$. Hypothesis \ref{(KU4)} then allows us to write 
\begin{align}\label{KU:eqn 6}
	\absval{\left(f_{0,a,U,j}^{-1}\right)'(z)}\comp\absval{z}^{-\frac{m(a)+1}{m(a)}}\absval{a}^{-\frac{\bt}{m(a)}}
\end{align}
for $z\in U$. Let $K\geq 1$ be the comparability constant for the previous equation \eqref{KU:eqn 6}. For two poles $a_1, a_2\in B_{2R_2}$ we denote by 
\begin{align*}
	f_{0,a_1,a_2,j}^{-1}:B(a_2,2R^*)\to\CC, \quad j=1\leq j\leq m(a_1),
\end{align*}
all inverse branches of $f_0$ which send the point $a_2$ to $a_1$. Considering \eqref{B_a(R)sub} and \eqref{KU:eqn 6} it then follows that 
\begin{align}\label{f0 inv R0 to S}
	f_{0,a_1,a_2,j}^{-1}(B(a_2,R^*))\sub B_{a_1}(2R_2-R^*)\sub B_{a_1}(R_2)\sub B(a_1, S)\sub B(a_1, R^*).
\end{align} 
Let $\seq{c_n}$ be a sequence in $\CC$ such that $\absval{c_n}<\ep$ for all $n\in\NN$, and define $f_n(z)=f_0(z)+c_n$ for all $n\in\NN$ and $z\in\CC$. Furthermore, let the function $F_+:\hat
\CC\to\hat{\CC}$ be defined by 
\begin{align*}
	F_+^n(z)=f_n\circ\dots\circ f_1(z)
\end{align*}
for all $n\in\NN$. By our choice of $S$ and $\ep>0$ we have that $z-c_n\in B(a,S^*)\sub B(a,R^*)$ for all $z\in B(a,2S)$ and $n\in\NN$. Thus, for poles $a_1,a_2\in B_{2R_2}$ the inverse branches $f_{n,a_1,a_2,j}^{-1}:B(a_2,2S)\to\CC$, $1\leq j\leq m(a_1)$, are well defined and given by 
\begin{align*}
	f_{n,a_1,a_2,j}^{-1}(z)=f_{0,a_1,a_2,j}^{-1}(z-c_n)
\end{align*} 
for $z\in B(a_2,2S)$. Moreover, in view of \eqref{f0 inv R0 to S}, we have that 
\begin{align}\label{fn sends S to S}
	f_{n,a_1,a_2,j}^{-1}(B(a_2,S))\sub B(a_1,S)
\end{align}
for each $n\in\NN$ and $1\leq j\leq m(a_1)$. Set 
\begin{align*}
	I_R(F_+):=\set{z\in\CC: \absval{F_+^n(z)}>R \text{ for all }n\geq 1}.
\end{align*}
Since $\sum_{a\in\cP}\absval{a}^{-u}$ converges if $u>\rho$, then given $t>\frac{\rho M}{\bt+M+1}$, there is $R_3\geq R_2$ sufficiently large such that 
\begin{align}\label{KU eqn 9}
	MK^t\sum_{a\in\cP_{R_3}}\absval{a}^{-t\frac{\bt+M+1}{M}}\leq 1.
\end{align}
Let $R_4>4R_3$ and define $I:=\cP_{R_3}$. Now, in view of \eqref{f0 inv R0 to S} and \eqref{fn sends S to S}, it follows that for every $n\in\NN$ and $R>2R_4$ the family of sets 
\begin{align*}
W_n=\set{f_{1,a_0,a_{1},j_0}^{-1}\circ\dots\circ f_{n,a_{n-1},a_n,j_{n-1}}^{-1}(B_{a_n}(R/2)):a_i\in I, 1\leq j_i\leq m(a_i), i=0,\dots,n}
\end{align*}
is well defined and covers $I_R(F_+)$. 
To see this we note that since $\infty$ is not an asymptotic value for $f_n$, each of the connected components of the inverse images of $B_R$ under $f_n$ contain neighborhoods of poles. 

In light of \eqref{diam ineq with S} and \eqref{KU:eqn 6}, we can write the following estimate 
\begin{alignat*}{2}
	\Sg_n &=&&
	\sum_{a_0\in I}\sum_{j_0=1}^{m(a_0)}\dots\sum_{a_{n-1}\in I}\sum_{j_{n-1}=1}^{m(a_{n-1})}\sum_{a_n\in I}\diam^t(f_{1,a_0,a_{1},j_0}^{-1}\circ\dots\circ f_{n,a_{n-1},a_n,j_{n-1}}^{-1}(B_{a_n}(R/2)))	\\
	&\leq&&
	\sum_{a_0\in I}\sum_{j_0=1}^{m(a_0)}\dots\sum_{a_{n-1}\in I}\sum_{j_{n-1}=1}^{m(a_{n-1})}\sum_{a_n\in I}\norm{(f_{1,a_0,a_{1},j_0}^{-1}\circ\dots\circ f_{n,a_{n-1},a_n,j_{n-1}}^{-1})'\rvert_{B_{a_n}(R/2)}}^t_\infty\diam^t(B_{a_n}(R/2))\\
	&\leq&&
	\sum_{a_0\in I}\sum_{j_0=1}^{m(a_0)}\dots\sum_{a_{n-1}\in I}\sum_{j_{n-1}=1}^{m(a_{n-1})}\sum_{a_n\in I}K^{nt} \left(\frac{\absval{a_{1}}^{-(m(a_0)+1)/m(a_0)}}{\absval{a_0}^{\bt/m(a_0)}}\right)^t\cdots\left(\frac{\absval{a_{n}}^{-(m(a_{n-1})+1)/m(a_{n-1})}}
	{\absval{a_{n-1}}^{\bt/m(a_{n-1})}}\right)^t\\
	& &&\times L^t\left(\frac{S}{2}\right)^{-\frac{t}{m(a_n)}}\frac{1}{\absval{a_n}^{t(\bt/m(a_n))}}\\
	&\leq&&
	L^t\left(\frac{2}{S}\right)^{\frac{t}{M}}K^{nt}\sum_{a_0\in I}\sum_{j_0=1}^{m(a_0)}\dots\sum_{a_{n-1}\in I}\sum_{j_{n-1}=1}^{m(a_{n-1})}\sum_{a_n\in I}\absval{a_0}^{-t(\bt/M)}\left(\absval{a_{n}}\cdots\absval{a_1}\right)^{-t\frac{\bt+M+1}{M}}\\
	&=&& 
	L^t\left(\frac{2}{S}\right)^{\frac{t}{M}}K^{nt}\sum_{a_0\in I}\sum_{j_0=1}^{m(a_0)}\dots\sum_{a_{n-1}\in I}\sum_{j_{n-1}=1}^{m(a_{n-1})}\sum_{a_n\in I}\left(\absval{a_{n}}\cdots\absval{a_0}\right)^{-t\frac{\bt+M+1}{M}}\\
	&\leq&&
	L^t\left(\frac{2}{S}\right)^{\frac{t}{M}}K^{nt}\left(\sum_{a\in I}\absval{a}^{-t\frac{\bt+M+1}{M}}\right)^nM^n\\
	&=&&
	L^t\left(\frac{2}{S}\right)^{\frac{t}{M}}\left(MK^t\sum_{a\in I}\absval{a}^{-t\frac{\bt+M+1}{M}}\right)^n.	
\end{alignat*}
Thus, \eqref{KU eqn 9} gives us that $\Sg_n\leq L^t(2/S)^{t/M}$ for each $n\in\NN$. Since the diameters of the sets of the covers $W_n$ converge to 0 uniformly as $n\to\infty$, we can estimate the $t$-dimensional Hausdorff measure to be
\begin{align*}
	H^t(I_R(F_+))\leq  L^t(2/S)^{t/M}.
\end{align*} 
Thus, we must have $\HD(I_R(F_+))\leq t$. Setting 
\begin{align*}
	I_{R,e}(F_+):=\set{z\in \CC:\liminfty{n}\absval{F_+^n(z)}>R}=\union_{n\geq 1}F^{-n}_+(I_R(F_+)),
\end{align*}
we see that 
\begin{align*}
	\HD(I_\infty(F_+))\leq\HD(I_{R,e}(F_+))=\HD(I_R(F_+))\leq t.	
\end{align*}
Letting $t\to\frac{\rho M}{\bt+M+1}$ provides the desired result.
\end{proof}
Together with Theorem \ref{UB of dimension of escaping set} the following theorem completes the proof of the first part of Theorem \ref{main thm: escaping set KU}.
\begin{theorem}\label{thm: KU f I_infty HD UB}
	If $f_0$ satisfies the same hypotheses \ref{(KU1)}-\ref{(KU4)} as in the previous theorem, then 
	\begin{align*}
		\HD(I_\infty(F_+))\geq\frac{\rho M^*}{\bt+M^*+1}.
	\end{align*}
\end{theorem}
\begin{proof}
	In order to prove Theorem \ref{thm: KU f I_infty HD UB} we follow the insights of Remark \ref{rem: escapers stay close to poles} in order to construct a non-stationary NCIFS which is contained in $I_\infty(F_+)$.  
	Let $R_0,\dots,R_4, S^*,S, \ep$ be as in the previous proof. Then, for $\seq{c_n}$ in $\CC$ with $\absval{c_n}<\ep$ and two poles $a,b\in B_{2R_2}$ we have 
	\begin{align}\label{fn sends S to S ascending}
		f_{n,a,b,1}^{-1}(\ol{B}(b,S))\sub \ol{B}(a,S).
	\end{align}
	Enumerate the set 
	\begin{align*}
		\sP:=\cP_{2R_2}\cap m^{-1}(M^*)=\set{a_0,a_1,\dots}	
	\end{align*}
	in such a way that $\absval{a_n}\leq \absval{a_{n+1}}$ for each $n\in\NN$. We may assume without loss of generality that $\absval{a_0}>1$, as if this were not the case we could simply increase the value of $R_2$. Recursively define a sequence $(\xi_n)_{n=0}^\infty$ of natural numbers as follows. Let $\xi_0=0$. Since $\sum_{a\in\cP}\absval{a}^{-u}$ converges if $u>\rho$, then for a fixed $t<\frac{\rho M^*}{\bt+M^*+1}$, for $n\geq 1$ let the number $\xi_{n,t}=\xi_n$, depending on $t$, be the least integer such that
	\begin{align}\label{sum bigger 2K}
		\sum_{j=\xi_{n-1}+1}^{\xi_n}\absval{a_j}^{-t\frac{\bt+M^*+1}{M^*}}\geq 2K^{\frac{4\rho M^*}{\bt+M^*+1}}\absval{a_n}^{2\left(\frac{\rho M^*}{\bt+M^*+1}\right)\left(\frac{\bt+M^*+1}{M^*}\right)}=2K^{\frac{4\rho M^*}{\bt+M^*+1}}\absval{a_n}^{2\rho},
	\end{align}
	where $K$ is defined as before, to be the constant of comparability coming from \eqref{KU:eqn 6}. As $\absval{a_n}\to\infty$ as $n\to\infty$ we see that $\xi_n\to\infty$ as $n\to\infty$ as well. For $n\geq 1$ define  
	\begin{align*}
		\gm_n:=\xi_{n+1}-\xi_n,
	\end{align*}
	let $\al_1=1$, and for $n\geq 2$ let
	\begin{align}\label{sums of betas}
	\al_n:=\sum_{j=2}^n\gm_j.
	\end{align}	
	Without loss of generality, we may assume that $\gm_1=\xi_1>1$, otherwise we may increase $R_2$ to be sufficiently large. Now we seek to define a NCIFS whose limit set sits inside of the set of escaping points. To that end, we begin by defining the alphabets on which our system operates. First we let 
	\begin{align*}
		I^{(\al_1)}=I^{(1)}=\set{1,\dots,\xi_1,\xi_1+1}.
	\end{align*}
	For each $1\leq j<\gm_2-1$ we let 
	\begin{align*}
		I^{(\al_1+j)}=I^{(1+j)}=\set{1,\dots,\xi_1,\xi_1+1,\dots,\xi_1+j}.
	\end{align*}	
	Then, for $j=\gm_2-1$ we have 
	\begin{align*}
		I^{(\al_2)}=I^{(\gm_2)}=\set{\xi_1+1,\dots,\xi_2}.
	\end{align*}
	In general, for $k\geq 2$ let 
	\begin{align*}
		I^{(\al_k)}=\set{\xi_{k-1}+1,\dots,\xi_k}
	\end{align*}
	and for $\al_k<n<\al_{k+1}$ with $n=\al_{k}+j$ for some $1\leq j\leq\gm_{k+1}-1$ let
	\begin{align*}
		I^{(n)}=I^{(\al_k+j)}=\set{\xi_{k-1}+1,\dots,\xi_k,\xi_k+1,\dots,\xi_k+j}.
	\end{align*}		
	Since the alphabets $I^{(n)}$ grow in size by at most one element with each time step, we clearly have that the alphabets grow subexponentially, that is
	\begin{align*}
		\lim_{n\to\infty}\frac{1}{n}\log\#I^{(n)}=0.
	\end{align*}
	With our alphabets defined we now wish to define a NCIFS over these alphabets which follows the orbit of $F_+$. To do this we first define the following counting function $T(n)$ which counts the number of iterations up to and through time $n$. Specifically, we let 
	\begin{align*}
		T(n)=2n+k  \;\text{ for }\; n=\al_k+j, 
	\end{align*} 
	where $0\leq j\leq\gm_{k+1}^*$ and  
	\[ 
	\gm_k^*=
	\begin{cases} 
	\gm_2-2\quad  & \text{ if }k=2 \\
	\gm_k-1 & \text{ if }k>2.  
	\end{cases}
	\]
	Now we define our contractions as follows. For $n=\al_k$, $k\geq 1$, and each $i\in I^{(n)}$ define the map $\phi_i^{(n)}:\ol{B}(a_k,S)\to\ol{B}(a_{k-1},S)$ by 
	\begin{align*}
		\phi_i^{(n)}:=f^{-1}_{T(n)-2,a_{k-1},a_{k},1}\circ f^{-1}_{T(n)-1,a_{i},a_{k},1}\circ f^{-1}_{T(n),a_{i},a_{k},1} 
	\end{align*}
	If $\al_k<n<\al_{k+1}$ with $n=\al_{k}+j$ for some $1\leq j\leq\gm_{k+1}-1$, then for each $i\in I^{(n)}$ we define the map $\phi_i^{(n)}:\ol{B}(a_k,S)\to\ol{B}(a_{k},S)$ by 
	\begin{align*}
	\phi_i^{(n)}:= f^{-1}_{T(n)-1,a_{i},a_{k},1}\circ f^{-1}_{T(n),a_{i},a_{k},1} 
	\end{align*}
	For each $n\in\NN$, with $\al_k\leq n<\al_{k+1}$, we define the functions $s,t$ as follows
	\[ 
	s(n)=
	\begin{cases} 
	k-1\quad  & \text{ if }n=\al_k \\
	k & \text{ if }\al_k<n<\al_{k+1}
	\end{cases}
	\]
	and
	\[
		t(n)=
			\begin{cases} 
			0\quad  & \text{ if }n=\al_1=1 \\
			k & \text{ if }1<\al_k<n<\al_{k+1}.  
			\end{cases}
	\]
	Letting
	\begin{align*}
		\Phi^{(n)}=\set{\phi_i^{(n)}:\ol{B}(a_{t(n)},S)\to\ol{B}(a_{s(n)},S): i\in I^{(n)}}
	\end{align*}
	denote the collection of contraction mappings at time $n$ on the sequence of closed disks $\left(\ol{B}(a_{t(n)},S)\right)_{n=0}^\infty$ defines a non-stationary NCIFS $\Phi$.	 
	Indeed, since each of the sets $\ol{B}(a_n,S)$ is convex, we have that the Uniform Cone, and Geometry Conditions are immediately satisfied. Furthermore the diameters are constant so we also have that the Diameter Condition is immediately satisfied as well. 
	By construction, we have that 
	\begin{align*}
		J_\Phi:=\intersect_{n=1}^\infty\union_{\om\in I^n}\phi_\om(\ol{B}(a_{t(n)},S))\sub I_\infty(F_+)\cap \ol{B}(a_0,S)
	\end{align*}
	since every pole is eventually discarded from the construction in favor of a pole of higher modulus. Furthermore, we have that $\Phi$ is subexponentially bounded. Thus, Bowen's formula holds. In order to find a lower bound for $\HD(I_\infty(F_+))$ we aim to find a lower bound for $B_\Phi=\HD(J_\Phi)$, which we accomplish by estimating $Z_{(n)}(t)$. We first consider the case when $n=\al_k$ for $k\geq 1$. Applying \eqref{sum bigger 2K}, in this case we have  	
	\begin{align}
		Z_{(n)}(t)&=\sum_{i\in I^{(n)}}\norm{\left(\phi_{i}^{(n)}\right)'}^t=\sum_{i=\xi_{k-1}}^{\xi_k}\norm{\left(f^{-1}_{T(n)-2,a_{k-1},a_{k},1}\circ f^{-1}_{T(n)-1,a_{i},a_{k},1}\circ f^{-1}_{T(n),a_{i},a_{k},1} \right)'}^t\nonumber\\
		&\geq K^{-3t}\sum_{i=\xi_{k-1}}^{\xi_k}\absval{a_k}^{-t\frac{M^*+1}{M^*}}\absval{a_{k-1}}^{-t\frac{\bt}{M^*}}\absval{a_k}^{-t\frac{\bt+M^*+1}{M^*}}\absval{a_i}^{-t\frac{\bt+M^*+1}{M^*}}\nonumber\\
		&\geq K^{-3t}\sum_{i=\xi_{k-1}}^{\xi_k}\absval{a_k}^{-2t\frac{\bt+M^*+1}{M^*}}\absval{a_i}^{-t\frac{\bt+M^*+1}{M^*}}\nonumber\\
		&\geq 2K^{\frac{3\rho M^*}{\bt+M^*+1}}K^{-3t}\absval{a_k}^{-2t\frac{\bt+M^*+1}{M^*}}\absval{a_k}^{2\rho}\nonumber\\
		&= 2K^{\frac{4\rho M^*}{\bt+M^*+1}-3t}\absval{a_k}^{2\left(\frac{\rho M^*}{\bt+M^*+1}-t\right)\frac{\bt+M^*+1}{M^*}}\label{eqn: Zn estimate NCGDMS 1}
	\end{align}
	For $\al_k< n<\al_{k+1}$ with $n=\al_k+j$ for $1\leq j\leq\gm^*_{k+1}$, again using \eqref{sum bigger 2K}, we similarly get
	\begin{align}
		Z_{(n)}(t)&=\sum_{i\in I^{(n)}}\norm{\left(\phi_{i}^{(n)}\right)'}^t=\sum_{i=\xi_{k-1}}^{\xi_k+j}\norm{\left( f^{-1}_{T(n)-1,a_{i},a_{k},1}\circ f^{-1}_{T(n),a_{i},a_{k},1} \right)'}^t\nonumber\\
		&\geq\sum_{i=\xi_{k-1}}^{\xi_k}\norm{\left( f^{-1}_{T(n)-1,a_{i},a_{k},1}\circ f^{-1}_{T(n),a_{i},a_{k},1} \right)'}^t\nonumber\\
		&\geq K^{-2t}\sum_{i=\xi_{k-1}}^{\xi_k}\absval{a_k}^{-t\frac{\bt+M^*+1}{M^*}}\absval{a_i}^{-t\frac{\bt+M^*+1}{M^*}}\nonumber\\
		&\geq 2K^{\frac{4\rho M^*}{\bt+M^*+1}}K^{-2t}\absval{a_k}^{-t\frac{\bt+M^*+1}{M^*}}\absval{a_k}^{2\rho}\nonumber\\
		&= 2K^{\frac{4\rho M^*}{\bt+M^*+1}-2t}\absval{a_k}^{\left(\frac{2\rho M^*}{\bt+M^*+1}-t\right)\frac{\bt+M^*+1}{M^*}}\nonumber\\
		&\geq 2K^{\frac{4\rho M^*}{\bt+M^*+1}-3t}\absval{a_k}^{2\left(\frac{\rho M^*}{\bt+M^*+1}-t\right)\frac{\bt+M^*+1}{M^*}}.	\label{eqn: Zn estimate NCGDMS 2}		
	\end{align}	
	Thus, in view of \eqref{eqn: Zn estimate NCGDMS 1}, \eqref{eqn: Zn estimate NCGDMS 2}, and the definition of $t(n)$, we have that for any $\al_k\leq n<\al_{k+1}$
	\begin{align}
		Z_{(n)}(t)\geq 2K^{\frac{4\rho M^*}{\bt+M^*+1}-3t}\absval{a_{t(n)}}^{2\left(\frac{\rho M^*}{\bt+M^*+1}-t\right)\frac{\bt+M^*+1}{M^*}}.	\label{eqn: Zn estimate NCGDMS}		
	\end{align}
	Applying \eqref{eqn: Zn ineq} and \eqref{eqn: Zn estimate NCGDMS} we see that 
	\begin{align*}
		Z_n(t)&\geq K^{-nt}Z_{(1)}(t)\cdots Z_{(n)}(t)\\
		&\geq 2^nK^{-nt}K^{n\left(\frac{4\rho M^*}{\bt+M^*+1}-3t\right)}\prod_{i=1}^n\absval{a_{t(i)}}^{2\left(\frac{\rho M^*}{\bt+M^*+1}-t\right)\frac{\bt+M^*+1}{M^*}}\\
		&\geq 2^nK^{4n\left(\frac{\rho M^*}{\bt+M^*+1}-t\right)}\absval{a_{t(n)}}^{2n\left(\frac{\rho M^*}{\bt+M^*+1}-t\right)\frac{\bt+M^*+1}{M^*}}.
	\end{align*}
	Thus, noting that $\absval{a_k}>1$ and $K\geq 1$, for $t<\frac{\rho M^*}{\bt+M^*+1}$ we have that $Z_n(t)\geq 2^n$. Consequently, $\ul{P}(t)>0$, which implies that 
	\begin{align*}
		t\leq\HD(J_\Phi)\leq \HD(I_\infty(F_+)). 
	\end{align*}
	Letting $t\to\frac{\rho M^*}{\bt+M^*+1}$ finishes the proof. 
\end{proof}
\begin{remark}
	We should point out that although we have chosen to present Theorem \ref{thm: KU f I_infty HD UB} within the generality of non-autonomous dynamics, the previous result, to the best of the author's knowledge, was not previously known even in the autonomous case.
\end{remark}
The following theorem differs from the previous two in two main ways. First, we no longer require that $\infty$ is not an asymptotic value for $f_0$, but rather, we will only require that $f_0$ has infinitely many poles. Second,  
our choice of the perturbative values $\ep,\dl$ will depend upon the value of $t$, and will inhibit our ability to find a satisfactory lower bound except in the case that $f_0$ is of divergence type. In particular, we prove the following.
\begin{theorem}\label{thm: UB KU hyp dim affine}
	Let $f_0:\CC\to\hat{\CC}$ be a transcendental meromorphic function with infinitely many poles that satisfies hypotheses \ref{(KU2)}-\ref{(KU4)}.
	 Then, for each $0\leq t<\frac{\rho M^*}{\bt +M^*+1}$ there exist $\ep_t,\dl_t>0$ such that if $c_t=\seq{c_n}$ and $\lm_t=\seq{\lm_n}$ are sequences in $\CC$ such that $\absval{c_n}<\ep_t$ and $\lm_n,\lm_n^{-1}\in B(1,\dl_t)$ for each $n\in\NN$, then
	 \begin{align*}
		 \HD(\jl_{r}(F_{A,\lm_t,c_t}))\geq t.
	 \end{align*}
	 If, in addition, $f_0$ is of divergence type, then there exists $\ep,\dl>0$, no longer depending on $t$, such that if $c=\seq{c_n}$ and $\lm=\seq{\lm_n}$ are sequences in $\CC$ such that $\absval{c_n}<\ep$ and $\lm_n,\lm_n^{-1}\in B(1,\dl)$ for each $n\in\NN$, then
	 \begin{align*}
		 \HD(\jl_{r}(F_{A,\lm,c}))\geq \frac{\rho M^*}{\bt +M^*+1}.
	 \end{align*}
\end{theorem}
\begin{proof}
With the exception of the choice of $\ep$, the proof runs the same as the proof of Theorem \ref{UB of dimension of escaping set} up to \eqref{f0 inv R0 to S}, i.e. let $R_0,\dots, R_4, S^*,S$ be the same such that we have 
\begin{align}\label{subset show fn sends S to S}
	f_{0,a_1,a_2,j}^{-1}(B(a_2,R^*))\sub B_{a_1}(2R_2-R^*)\sub B_{a_1}(R_2)\sub B(a_1, S)\sub B(a_1, R^*).
\end{align} 
Again, let
\begin{align*}
	\sP:=\cP_{2R_2}\cap m^{-1}(M^*)=\set{a_0,a_1,\dots}	
\end{align*}
be enumerated, such that $\absval{a_n}\leq\absval{a_{n+1}}$ for all $a_n\in\sP$ and all $n\geq 0$, and again we assume that $R_2$ has been taken large enough such that $\absval{a_0}>1$. Now, since $\sum_{a\in\cP}\absval{a}^{-u}$ converges if $u>\rho$, then for $t<\frac{\rho M^*}{\bt+M^*+1}$, there is some $N_t\in\NN$, depending on $t$, such that 
\begin{align}\label{KU sum diverge}
	\sum_{n=1}^{N_t}\absval{a_n}^{-t\frac{\bt+M^*+1}{M^*}}\geq 2K^{3\cdot\left(\frac{\rho M^*}{\bt+M^*+1}\right)}\absval{a_0}^{\frac{\rho M^*}{\bt+M^*+1}\cdot\frac{\bt+M^*+1}{M^*}}=2K^{3\cdot\left(\frac{\rho M^*}{\bt+M^*+1}\right)}\absval{a_0}^{\rho}.
\end{align}
Let $I=\set{a_1,\dots, a_{N_t}}$. Choose $\ep_t,\dl_t>0$ such that the following hold"
\begin{itemize}
	\item $\ep_t<\dl_t$,
	\item $\dl_t<\frac{S^*-2S}{2S}$,
	\item $(1+\dl_t)(\dl_t(1+\absval{a})+S)<\frac{S^*}{2}$ for all $a\in I$.
\end{itemize}
Let $\seq{c_n}$ and $\seq{\lm_n}$ be sequences in $\CC$ such that $\absval{c_n}<\ep_t$ and $\lm_n,\lm_n^{-1}\in B(1,\dl_t)$ for each $n\in\NN$ and define $f_n:\CC\to\hat{\CC}$ to be the affine perturbation of $f_0$ at time $n$ given by 
\begin{align*}
	f_n(z)=\lm_nf_0(z)+c_n.
\end{align*} 
By our choice of $\ep_t,\dl_t$ we have that for each $a\in I$ and each $z\in B(a,S)$ 
\begin{align}\label{ineq show fn maps S to S}
	\frac{z-c_n}{\lm_n}\in B(a,S^*).
\end{align}
Indeed, 
\begin{align}
	\absval{\frac{z-c_n}{\lm_n}-a}&\leq\absval{\lm_n^{-1}}\left(\absval{c_n}+\absval{z-a}+\absval{1-\lm_n}\absval{a}\right)\nonumber\\
	&\leq(1+\dl_t)\left(\ep_t+S+\dl_t\absval{a}\right)\nonumber\\
	&\leq(1+\dl_t)\left(\dl_t(1+\absval{a})+S\right)<\frac{R_0}{2}.\label{ineq to show dl exists}
\end{align}
The requirement that $\dl_t<\frac{S^*-2S}{2S}$ ensures that such a $\dl_t$ exists. As it implies that 
$$
	(1+\dl_t)S<S^*/2,
$$
we see that solving \eqref{ineq to show dl exists} reduces to choosing
\begin{align*}
	0<\dl_t<\frac{-1+\sqrt{1+2S^*(1+\absval{a_{N_t}})^{-1}}}{2}.
\end{align*}
For each $a\in I$ we fix inverse branches of $f_n$
\begin{align*}
	f_{n,a,a_0,1}^{-1}:\ol{B}(a,S)\to\CC \spand f_{n,a_0,a,1}^{-1}:\ol{B}(a_0,S)\to\CC.
\end{align*}
Together \eqref{subset show fn sends S to S} and \eqref{ineq show fn maps S to S} gives us that 
\begin{align*}
	f_{n,a,a_0,1}^{-1}(\ol{B}(a,S))\sub\ol{B}(a_0,S) \spand 
	f_{n,a_0,a,1}^{-1}(\ol{B}(a_0,S))\sub\ol{B}(a,S).
\end{align*}
For each $n\in\NN$ and $a\in I$ we let the function $\phi_a^{(n)}$ be defined by 
\begin{align*}
	\phi_a^{(n)}:=f_{2n-1,a_0,a,1}^{-1}\circ f_{2n,a,a_0,1}^{-1}:\ol{B}(a_0,S)\to\ol{B}(a_0,S).
\end{align*} 
Then, each of the functions $\phi_a^{(n)}$ is a contraction, and as there are only finitely many of them, they are in fact uniformly contracting. 
Thus, the collection 
\begin{align*}
	\Phi=\left(\Phi^{(n)}\right)_{n\in\NN}=\left(\set{\phi_a^{(n)}:a\in I}\right)_{n\in\NN}
\end{align*}
forms a stationary NCIFS in the style of \cite{rempe-gillen_non-autonomous_2016}, for which Bowen's formula holds.
The limit set $J_\Phi$ of the NCIFS $\Phi$ is given by
\begin{align*}
	J_\Phi=\intersect_{n=1}^\infty\union_{\om\in I^n}\phi_\om(\ol{B}(a_0,S)).
\end{align*}
As $\absval{(\phi_\om)'(z)}\to0$ for $\absval{\om}=n\to\infty$, we have that $\absval{(F_A^{2n})'(z)}\to\infty$ as $n\to\infty$. Thus, Lemma \ref{lem: deriv to infty implies Julia} implies that $J_\Phi\sub\jl(F_A)$. 
In fact, by construction, we have that $J_\Phi\sub\cJ_r(F_A)$. Now, for each $n\in\NN$ we use \eqref{KU sum diverge} to estimate 
\begin{align*}
	Z_{(n)}(t)&=\sum_{a\in I}\norm{(\phi_a^{(n)})'}^t=\sum_{a\in I}\norm{(f_{2n-1,a_0,a,1}^{-1}\circ f_{2n,a,a_0,1}^{-1})'}^t\\
	&\geq K^{-2t}\sum_{a\in I}\absval{a_0}^{-t\frac{\bt+M^*+1}{M^*}}\absval{a}^{-t\frac{\bt+M^*+1}{M^*}}\\
	&\geq 2K^{\frac{3\rho M^*}{\bt+M^*+1}-2t}\absval{a_0}^{-t\frac{\bt+M^*+1}{M^*}}\cdot \absval{a_0}^\rho\\
	&=2K^{\frac{3\rho M^*}{\bt+M^*+1}-2t}\absval{a_0}^{\left(\frac{\rho M^*}{\bt+M^*+1}-t\right)\frac{\bt+M^*+1}{M^*}}.
\end{align*}
Since $t<\frac{\rho M^*}{\bt+M^*+1}$, $K\geq 1$, and $\absval{a_0}\geq 1$ for each $n\in\NN$, we have 
\begin{align*}
	Z_n(t)\geq K^{-nt}Z_{(1)}(t)\cdots Z_{(n)}(t)\geq 2^nK^{3n\left(\frac{\rho M^*}{\bt+M^*+1}-t\right)}\absval{a_0}^{n\left(\frac{\rho M^*}{\bt+M^*+1}-t\right)\frac{\bt+M^*+1}{M^*}}\geq 2^n.
\end{align*}
Thus, $\ul{P}(t)>0$ and hence, $\HD(J_\Phi)\geq t$, which finishes the proof of the first statement. 

Now, if $f_0$ is of divergence type, then for $t=\frac{\rho M^*}{\bt+M^*+1}$ we have that the sum
\begin{align*}
	\sum_{a_n\in\cP}\absval{a_n}^{-t\frac{\bt+M^*+1}{M^*}}=\infty,
\end{align*}
and as such, we are able to find $N_t<\infty$ as in \eqref{KU sum diverge}. Continuing the proof from there in the same manner as before, we see that there is $\ep,\dl>0$, which no longer depend on $t$, such that 
\begin{align*}
	\frac{\rho M^*}{\bt+M^*+1}\leq \HD(J_\Phi)\leq\HD(\jl_{r}(F_A)),
\end{align*} 
completing the proof. 
\end{proof}
\begin{remark}
	Note that our choices of $\ep,\dl$ must go to zero as $t$ approaches the critical exponent unless we know that the function $f_0$ is of divergence type. This is precisely because in the case where $f_0$ is of divergence type we are assured a finite number $N_t$ such that the sum \eqref{KU sum diverge} is sufficiently large.
	If $f_0$ is not of divergence type, then we must choose $N_t$ equal to $\infty$, which necessarily means that the values $\ep,\dl$ must be equal to zero as they are tied to the value of $N_t$ in an inverse manner.
\end{remark}
\section{Eventual Dimensions}\label{sec ED}
In this section we collect together several results, some of which are new and some of which are already known, concerning the eventual dimension and the eventual hyperbolic dimension of several classes of transcendental functions. In particular, we provide results for the two main classes which have already been discussed. 

The \textit{eventual dimension} of a function $f$, given by
\begin{align*}
	\ED(f)=\limty{R}\HD(\set{z\in\cJ(f):\absval{f^n(z)}> R, \; \forall n\geq 1}),
\end{align*} 
was first introduced by Rempe-Gillen and Stallard for entire functions $f$ in \cite{rempe_hausdorff_2009}, though it had been used implicitly before by several authors. The definition, however, is equally valid in the case that $f$ is meromorphic.
The following proposition was proven by Rempe-Gillen and Stallard first in the case of transcendental entire functions, but their same proof holds more generally for transcendental meromorphic functions. 
\begin{proposition}
	Let $f$ be a transcendental meromorphic function. Then, 
\begin{align}\label{ED I infty ineq}
	\HD(I_\infty(f))\leq\ED(f)\leq\HD(\cJ(f)).
\end{align}
\end{proposition}
In \cite{bergweiler_hausdorff_2012}, Bergweiler and Kotus show that for a transcendental meromorphic function $f\in\cB$ of finite order $\rho$ such that $\infty$ is not an asymptotic value and there is some $M\in\NN$ such that the multiplicity of co-finitely many poles is at most $M$, then 
\begin{align*}
\HD(I_\infty(f))\leq\ED(f)\leq \frac{2M\rho}{2+M\rho}.
\end{align*}
In fact, they provide a function $f$ such that
\begin{align*}
\HD(I_\infty(f))=\frac{2M\rho}{2+M\rho}\spand \HD\left(\set{z\in\CC: \liminfty{n}\absval{f^n(z)}\geq R}\right)>\frac{2M\rho}{2+M\rho}
\end{align*} 
for all $R>0$. In particular, we see that there is a transcendental meromorphic function $f$ such that 
\begin{align*}
\HD(I_\infty(f))<\ED(f).
\end{align*}
This of course shows that the first inequality of \eqref{ED I infty ineq} may in fact be strict and the two quantities need not be equal.  

The notion of the \textit{eventual hyperbolic dimension} of a function $f$, which was introduced by De Zotti and Rempe-Gillen for entire functions, is given by 
\begin{align*}
	\EHD_1(f)=\sup\set{\HD(X): X\sub B_R\text{ is hyperbolic for }f},
\end{align*}
where the set $X\sub\CC$ is \textit{hyperbolic} for $f$ if $X$ is compact and forward invariant such that for some $n\in\NN$ and some $\lm>1$ we have
\begin{align*}
	\absval{(f^n)'\rvert_X}>\lm.
\end{align*} 
Again, this definition is valid for meromorphic functions. In \cite{rempe_hyperbolic_2009} Rempe-Gillen shows that the hyperbolic dimension of a function $f$ is the same as the Hausdorff dimension of its radial Julia set, i.e.
\begin{align*}
	\HypDim(f)=\HD(\cJ_r(f)).
\end{align*}
The same proof shows that the same relationship between the dimension of hyperbolic sets and the dimension of the radial Julia set is also true for the eventual hyperbolic dimension of a meromorphic function $f$. Indeed, we have the following.
\begin{theorem}
	Given a meromorphic function $f:\CC\to\hat{\CC}$, the quantities $\EHD_1(f)$ and 
	\begin{align*}
	\EHD_2(f):=\limty{R}\HD\left(\set{z\in\cJ_r(f):\absval{f^n(z)}>R,\;\forall n\geq 0}\right)
	\end{align*}
	exist and are equal. We call their common value the eventual hyperbolic dimension of $f$, and denote it by 
	\begin{align*}
	\EHD(f)=\EHD_1(f)=\EHD_2(f).
	\end{align*}
\end{theorem} 
\begin{remark}
	Notice that the notion of eventual dimension immediately generalizes to include all non-autonomous functions and even though the idea of a hyperbolic set is not clear for non-autonomous dynamics. In light of the previous theorem, we take the eventual hyperbolic dimension of a general non-autonomous function to be the Hausdorff dimension of its radial Julia set. 
\end{remark}
Clearly by definition, specifically the definition of $\EHD_2(f)$, we have that 
\begin{align}\label{EHD ineq}
	\EHD(f)\leq \ED(f) \spand \EHD(f)\leq\HD(\cJ_r(f)).
\end{align}
Together with \eqref{EHD ineq}, the following theorem completes the proof of Theorem \ref{main thm: escaping set KU}. Its proof follows from the proofs of Theorems \ref{UB of dimension of escaping set} and \ref{thm: KU f I_infty HD UB} by letting $R_2\to\infty$ as each proof relies on the construction of a NCIFS contained sufficiently well within $B_{R_2}$.
\begin{theorem}
	Suppose $f_0$ satisfies the hypotheses of Theorem \ref{UB of dimension of escaping set}. Then, there is $\ep>0$ such that if $\seq{c_n}$ is a sequence in $\CC$ with $\absval{c_n}<\ep$ for all $n\in\NN$, then 
	\begin{align*}
	\frac{\rho M^*}{\bt+M^*+1}\leq \EHD(F_+) \leq\ED(F_+)\leq\frac{\rho M}{\bt+M+1}.
	\end{align*}
\end{theorem}
The same alteration made to the proof of Theorem \ref{thm: UB KU hyp dim affine}, i.e. letting $R_2\to\infty$, gives the following theorem which together with \eqref{EHD ineq} and Theorem \ref{thm: UB KU hyp dim affine} finally completes the proof of Theorem \ref{main thm KU hyper}.
\begin{theorem}
	Let $f_0:\CC\to\hat{\CC}$ be a transcendental meromorphic function with infinitely many poles that satisfies hypotheses \ref{(KU2)}-\ref{(KU4)}.
	Then, for each $0\leq t<\frac{\rho M^*}{\bt +M^*+1}$ there exist $\ep_t,\dl_t>0$ such that if $c_t=\seq{c_n}$ and $\lm_t=\seq{\lm_n}$ are sequences in $\CC$ such that $\absval{c_n}<\ep_t$ and $\lm_n,\lm_n^{-1}\in B(1,\dl_t)$ for each $n\in\NN$, then
	\begin{align*}
		\EHD(F_{A,\lm_t,c_t})\geq t.
	\end{align*}
	If, in addition, $f_0$ is of divergence type, then there exists $\ep,\dl>0$, no longer depending on $t$, such that if $c=\seq{c_n}$ and $\lm=\seq{\lm_n}$ are sequences in $\CC$ such that $\absval{c_n}<\ep$ and $\lm_n,\lm_n^{-1}\in B(1,\dl)$ for each $n\in\NN$, then
	\begin{align*}
		\EHD(F_{A,\lm,c})\geq \frac{\rho M^*}{\bt +M^*+1}.
	\end{align*}
\end{theorem}
Furthermore,  \eqref{EHD ineq} along with Theorems \ref{main thm: escaping set KU} and \ref{main thm KU hyper} proves Corollary~\ref{cor: summary}.
The following theorem was mentioned briefly in \cite{bergweiler_hausdorff_2012} as a consequence of Mayer's technique from \cite{mayer_size_2007}, though no formal 
proof was given. We now give a short proof of the following theorem, which along with Theorem \ref{thm: Mayer functions}, completes the proof of Theorem \ref{main thm: Mayer functions}.

\begin{theorem}
		Let $f$ be a meromorphic function of finite order $\rho$ and suppose that the following hold. 
		\begin{enumerate}
			\item $f$ has infinitely many poles $b_i\in f(\CC)$ with $b_i\not\in\ol{\Sing(f^{-1})}$ for each $i\geq 1$. Suppose that $m(b_i)\leq q<\infty$ for each $i\geq 1$.
			\item There are uniform constants $s>0$, $Q>0$ and $\al>-1-1/q$ such that for each $i\in\NN$ 
			\begin{align*}
			\absval{f'(z)}\leq Q\absval{z}^{\al} \qquad for \qquad z\in f^{-1}(U_i), \absval{z}\to\infty,
			\end{align*} 
			where $U_i=B(b_i,s)$. 
		\end{enumerate}
		Then, 
		$$
			\EHD(f)\geq \frac{\rho}{\al+1+1/q}.
		$$ 
		If in addition $f$ is of divergence type, then this inequality is strict. 
\end{theorem}
\begin{proof}
	Since $\absval{b_i}\to\infty$ as $i\to\infty$, Theorem \ref{thm: Mayer functions} allows us to construct an autonomous iterated function system $\Phi_i$ contained in $\cJ_r(f)\cap B_{R_i}$, where $R_i=\absval{b_i}-2s$, such that 
	\begin{align*}
		\HD(J_{\Phi_i})\geq \frac{\rho}{\al+1+1/q}.
	\end{align*}
	Letting $i\to\infty$, and subsequently $R_i\to\infty$, finishes the proof of the first part. 
	
	Now, if $f$ is of divergence type then the IFS $\Phi_i$ is hereditarily regular (see \cite{mauldin_dimensions_1996, mauldin_graph_2003}) and it thus follows from Theorem~3.20 of \cite{mauldin_dimensions_1996} that we may sharpen our estimate so that
	\begin{align*}
		\HD(J_{\Phi_i})> \frac{\rho}{\al+1+1/q}.
	\end{align*}
	Again, letting $i\to\infty$ finishes the proof. 
\end{proof}
\begin{remark}
	The proof gives more. In fact, we see that for each $i\in\NN$ there exists $\ep_{i},\dl_{i}>0$ such that if $\lm_{i}=\seq{\lm_n}$ and $c_{i}=\seq{c_n}$ are sequences in $\CC$ with $\lm_n,\lm_n^{-1}\in B(1,\dl_{i})$ and $\absval{c_n}<\ep_{i}$ for each $n\in\NN$ then 
	\begin{align*}
		\HD(\cJ_r(F_{A,\lm_{i},c_{i}})\cap B_{R_i})\geq\frac{\rho}{\al+1+1/q},
	\end{align*}
	where $R_i=\absval{b_i}-2s$.
	
	It is worth noting that $\ep_{i}$ and $\dl_{i}$ depend on $\absval{b_i}$, and in particular we have that $\limty{i}\dl_{i}=0$
	as well as the respective statements for $\ep_{i}$. So, we are unable to find non-autonomous perturbations that work uniformly for each pole $b_i$, $i\in\NN$. 
\end{remark}

\section{Examples}\label{sec: Examples}
As the calculation of the perturbative values $\ep,\dl$ can be quite complicated, for each of the following examples we will instead show that the necessary hypotheses are satisfied in order to apply our theorems. 
\begin{example}[Periodic Functions]
The polynomial growth condition of \eqref{eqn: M derivative condition} is satisfied for every periodic function $f$ with $\al=0$. Therefore, we may apply Theorem \ref{main thm: Mayer functions} for any periodic function such that there exists a pole $b\not\in\ol{\Sing(f^{-1})}$. With additional information, such as the existence of infinitely many poles, we may apply Theorems \ref{main thm: escaping set KU} and \ref{main thm KU hyper}. The following example produces a class of such periodic functions. 
\end{example}
\begin{example}[Rational Exponentials]
	Let 
	\begin{align*}
	f(z)=R(e^z),
	\end{align*}
	where $R$ is a rational function such that $R(0)\neq \infty$ and $R(\infty)\neq \infty$. Then, $f$ is a simply periodic function with finitely many poles in each strip of periodicity. Furthermore, $\Sing(f^{-1})=\set{R(0),R(\infty)}$ and it is easy to check that we can apply Theorems \ref{main thm: escaping set KU} and \ref{main thm KU hyper} with $\rho=1$ and $\bt=0$. 
	In particular,
	\begin{align*}
	f_0(z)=\mu(\tan(z))^m, \quad m\in\NN \text{ and } \mu\in\CC^*
	\end{align*} 
	is such a function. Moreover, since each of the poles are of multiplicity $m$, we can find $\ep,\dl>0$ such that 
	\begin{align*}
		\HD(I_\infty(F_+))=\EHD(F_+)=\ED(F_+)=\frac{m}{m+1}\leq \HD(\cJ_r(F_+)),
	\end{align*}
	and 
	\begin{align*}
		\HD(\cJ_r(F_A))\geq \EHD(F_A)\geq \frac{m}{m+1}.
	\end{align*}
	This second inequality is precisely the inequality obtained for the autonomous case in \cite{kotus_hausdorff_1995,mayer_size_2007}. For autonomous dynamics, we can improve the inequality concerning the hyperbolic dimension. It follows from \cite{skorulski_existence_2006} that $\HD(\cJ_r(f_0))>1$, and we expect that something similar should hold in the non-autonomous case.	
\end{example}
\begin{example}[Elliptic Functions]
	Elliptic functions have been a subject of much study lately. Previously, the non-autonomous case of elliptic functions has been covered in \cite{atnip_nonautonomous_2017} while the autonomous and random cases have been discussed in \cite{mayer_size_2007,kotus_hausdorff_2003,kotus_hausdorff_2008,roy_random_2011}.
	
	If $f_0$ is an elliptic function, then, by definition, we have that there exists $w_1,w_2\in\CC$ with $\Im(\frac{w_1}{w_2})>0$, where $\Im(z)$ denotes the imaginary part of the complex number $z$, such that $f(z)=f(\zeta)$ if and only if $\zeta=z+nw_1+mw_2$ for some $n,m\in\ZZ$. Then, we have that $\rho=2$, $\bt=0$, and so applying Theorem \ref{main thm: escaping set KU}, we have that there exist $\ep,\dl>0$ such that 
	\begin{align*}
		\HD(I_\infty(F_+))=\EHD(F_+)=\ED(F_+)=\frac{2q}{q+1}\leq \HD(\cJ_r(F_+)),
	\end{align*}
	where $q$ is the maximum multiplicity of each of the poles of $f_0$. As $f_0$ is of divergence type, Theorem \ref{main thm KU hyper} gives that there exist $\ep,\dl>0$ such that 
	\begin{align*}
		\HD(\cJ_r(F_A))\geq \EHD(F_A)\geq \frac{2q}{q+1}.
	\end{align*}
\end{example}
\begin{example}[Exponential Elliptics]
In \cite{mayer_exponential_2005} Mayer and Urba\'nski show that the Julia set of a function of the form 
\begin{align*}
	f(z)=\mu e^{g(z)} \text{ for } \mu\in\CC^*,
\end{align*}
where $g$ is a non-constant elliptic function, has Hausdorff dimension equal to $2$. In \cite{mayer_size_2007}, Mayer shows that such functions also have hyperbolic dimension equal to $2$. The same is true for the non-autonomous case. To see this, we must first note that functions of this form do not satisfy the hypotheses of Theorem \ref{main thm: Mayer functions}. However, functions of the form 
\begin{align*}
	f^{(d)}_0(z)=\mu\left(1+\frac{g(z)}{d}\right)^d \text{ for } \mu\in\CC^*, d\in\NN,
\end{align*}
do satisfy the hypotheses of Theorem \ref{main thm: Mayer functions} with $\rho=2$, $\al=0$, and the maximum multiplicity of poles equal to $dq$, where $q$ is the maximum multiplicity of the poles of $g$. Then, we have that there exist $\ep,\dl>0$ such that if $\seq{\lm_n}$ and $\seq{c_n}$ are sequences in $\CC$ such that 
\begin{align*}
\lm_n,\lm_n^{-1}\in B(1,\dl)\spand \absval{c_n}<\ep
\end{align*}
for each $n\in\NN$, then    
\begin{align*}
	\HD(\cJ_r(F_A^{(d)}))\geq\EHD(\cJ_r(F_A^{(d)}))\geq \frac{2dq}{dq+1},
\end{align*}
where $F_A^{(d)}$ is the function of non-autonomous affine perturbations of $f_0^{(d)}$. 
As $\ep,\dl$ are independent of $d$, letting $d\to\infty$ we see that 
\begin{align*}
	\HD(\cJ_r(F_A))=\EHD(\cJ_r(F_A))=2,
\end{align*}
where $F_A$ is the function of non-autonomous affine perturbations of the exponential elliptic function $f$. 
\end{example}
\begin{example}[Polynomial Schwarzian Derivative]
Recall that the Schwarzian derivative of a function $f$ is given by 
\begin{align*}
	S(f)=\left(\frac{f''}{f'}\right)'-\frac{1}{2}\left(\frac{f''}{f'}\right)^2.
\end{align*}
Exponential and tangent functions are examples of classes which have constant Schwarzian derivative. Examples for which $S(f)$ is a polynomial are 
\begin{align*}
	f(z)=\int_0^z e^{Q(w)}dw, 
\end{align*}
where $Q(w)$ is a polynomial, and 
\begin{align*}
	f(z)=\frac{aA_i(z)+bB_i(z)}{cA_i(z)+dB_i(z)} \text{ with } ad-bc\neq 0,
\end{align*}
where $A_i,B_i$ are the Airy functions of the first and second kind respectively.
If $S(f_0)=P$, a polynomial of degree $d$, then one can show that it satisfies the hypotheses of Theorem \ref{main thm: Mayer functions} with $\rho=d/2+1$ and $\al=d/2$ and is even of divergence type (see Section 2.4 of \cite{mayer_thermodynamical_2007} for details). Applying Theorem \ref{main thm: Mayer functions} we have that there are $\ep,\dl>0$ such that 
\begin{align*}
	\HD(\jl_{r}(F_A))\geq \EHD(F_A)\geq \frac{d+2}{d+4}\geq \frac{1}{2}.
\end{align*}
\end{example}
\begin{example}
	Let
	\begin{align*}
		f_0(z)=\frac{1}{z\sin(z)}.
	\end{align*}
	Then, $f_0$ is a meromorphic function of order $\rho=1$ with infinitely many poles,
	\begin{align*}
		\cP=\set{n\pi: n\in\ZZ},
	\end{align*}
	all of which are simple except for $0$. We also have that $\infty$ is not an asymptotic value and the set of singular values consists of the lone asymptotic value $z=0$ and infinitely many critical values of the form $v_n\comp\pm\frac{2}{(2n+1)\pi}$ for $n\in\ZZ$, which implies that $f_0\in\cB$. One can then show that $\bt=1$ and that each pole has multiplicity equal to $1$. As $f_0$ is of divergence type, we may apply Theorems \ref{main thm: escaping set KU} and \ref{main thm KU hyper} to obtain that there exist $\ep,\dl>0$ such that 
	\begin{align*}
		\HD(I_\infty(F_+))=\EHD(F_+)=\ED(F_+)=\frac{1}{3}\leq \HD(\cJ_r(F_+))		
	\end{align*}
	and
	\begin{align*}
		\HD(\cJ_r(F_A))\geq \EHD(F_A)\geq \frac{1}{3}.		
	\end{align*}
\end{example}
\begin{example}
	Let 
	\begin{align*}
		f_0(z)=\frac{1}{z\cos(\sqrt{z})}.
	\end{align*}
	Then, $f_0$ is a meromorphic function with infinitely many poles
	\begin{align*}
		\cP=\set{0}\cup\set{\left(\frac{(2n+1)\pi}{2}\right)^2:n\in\NN},
	\end{align*}
	all of which are simple and have multiplicity identically equal to $1$. One can easily check that $\rho=1/2$, $\bt=1/2$ and that $f_0$ is of divergence type. Note that the singular values of $f_0$ contain a single asymptotic value, namely $z=0$, and infinitely many critical values $v_n$ such that 
	\begin{align*}
		\absval{v_n}=\absval{\frac{1}{(n\pi+(n\pi)^{-1})\cos(\sqrt{n\pi+(n\pi)^{-1}})}}.
	\end{align*} 
	As the right hand side tends towards $0$ as $n\to\infty$, we have that $f_0\in\cB$. Consequently, we are able to apply Theorems \ref{main thm: escaping set KU} and \ref{main thm KU hyper}. Thus, there exists $\ep,\dl>0$ such that 
	\begin{align*}
	\HD(I_\infty(F_+))=\EHD(F_+)=\ED(F_+)=\frac{1}{5}\leq \HD(\cJ_r(F_+))		
	\end{align*}
	and
	\begin{align*}
	\HD(\cJ_r(F_A))\geq \EHD(F_A)\geq \frac{1}{5}.		
	\end{align*}
\end{example}
\section{Acknowledgments}
We wish to thank Mariusz Urba\'nski and William Cherry, whose invaluable comments helped to improve this manuscript. We also wish to thank the comments to two anonymous referees whose comments helped to improve this manuscript, and in particular, led to an improved version of Theorem~\ref{main thm: Mayer functions}.

\bibliographystyle{jabbrv_plain}

\begin{thebibliography}{10}
	
	
	\bibitem{atnip_nonautonomous_2017}
	Jason Atnip.
	\newblock Non-autonomous conformal graph directed {Markov} systems.
	\newblock {\em\JournalTitle{arXiv:1706.09978}}, June 2017.
	
	\bibitem{bergweiler_iteration_1993}
	Walter Bergweiler.
	\newblock Iteration of meromorphic functions.
	\newblock {\em\JournalTitle{American Mathematical Society. Bulletin. New
			Series}}, 29(2):151--188, 1993.
	
	\bibitem{bergweiler_hausdorff_2012}
	Walter Bergweiler and Janina Kotus.
	\newblock On the {Hausdorff} dimension of the escaping set of certain
	meromorphic functions.
	\newblock {\em\JournalTitle{Transactions of the American Mathematical
			Society}}, 364(10):5369--5394, 2012.
	
	\bibitem{eremenko_dynamical_1992}
	A.~\'E. Er\"emenko and M.~Yu. Lyubich.
	\newblock Dynamical properties of some classes of entire functions.
	\newblock {\em\JournalTitle{Universit\'e de Grenoble. Annales de l'Institut
			Fourier}}, 42(4):989--1020, 1992.
	
	\bibitem{heinonen_lectures_2001}
	Juha Heinonen.
	\newblock {\em Lectures on analysis on metric spaces}.
	\newblock Springer, New York, 2001.
	
	\bibitem{kotus_hausdorff_1995}
	Janina Kotus.
	\newblock On the {Hausdorff} dimension of {Julia} sets of meromorphic
	functions. {II}.
	\newblock {\em\JournalTitle{Bulletin de la Soci\'et\'e; math\'ematique de
			France}}, 123(1):33--46, 1995.
	
	\bibitem{kotus_hausdorff_2003}
	Janina Kotus and Mariusz Urba\'nski.
	\newblock Hausdorff dimension and {H}ausdorff measures of {J}ulia sets of
	elliptic functions.
	\newblock {\em\JournalTitle{Bulletin of the London Mathematical Society}},
	35(02):269--275, March 2003.
	
	\bibitem{kotus_fractal_2008}
	Janina Kotus and Mariusz Urba\'nski.
	\newblock Fractal measures and ergodic theory of transcendental meromorphic
	functions.
	\newblock In {\em Transcendental dynamics and complex analysis}, volume 348 of
	{\em London {Math}. {Soc}. {Lecture} {Note} {Ser}.}, pages 251--316.
	Cambridge Univ. Press, Cambridge, 2008.
	\newblock DOI: 10.1017/CBO9780511735233.013.
	
	\bibitem{kotus_hausdorff_2008}
	Janina Kotus and Mariusz Urba\'nski.
	\newblock Hausdorff dimension of radial and escaping points for transcendental
	meromorphic functions.
	\newblock {\em\JournalTitle{Illinois Journal of Mathematics}}, 52(3):1035,
	2008.
	
	\bibitem{mauldin_dimensions_1996}
	R.~Daniel Mauldin and Mariusz Urba\'nski.
	\newblock Dimensions and measures in infinite iterated function systems.
	\newblock {\em\JournalTitle{Proceedings of the London Mathematical Society}},
	3(1):105--154, 1996.
	
	\bibitem{mauldin_graph_2003}
	R.~Daniel Mauldin and Mariusz Urba\'nski.
	\newblock {\em Graph directed {Markov} systems geometry and dynamics of limit
		sets}.
	\newblock Cambridge University Press, Cambridge, 2003.
	
	\bibitem{mayer_size_2007}
	Volker Mayer.
	\newblock The size of the {Julia} set of meromorphic functions.
	\newblock {\em\JournalTitle{Mathematische Nachrichten}}, 282(8):1189--1194,
	August 2009.
	
	\bibitem{mayer_thermodynamical_2007}
	Volker Mayer and Mariusz Urba\'nski.
	\newblock {\em Thermodynamical formalism and multifractal analysis for
		meromorphic functions of finite order}, volume 203 of {\em Memoirs of the
		{American} {Mathematical} {Society}}.
	\newblock American Mathematical Society, 2010.
	
	\bibitem{mayer_exponential_2005}
	Volker Mayer, Mariusz Urba\'nski, and {others}.
	\newblock Exponential elliptics give dimension two.
	\newblock {\em\JournalTitle{Illinois Journal of Mathematics}}, 49(1):291--294,
	2005.
	
	\bibitem{rempe_hyperbolic_2009}
	Lasse Rempe.
	\newblock Hyperbolic dimension and radial {Julia} sets of transcendental
	functions.
	\newblock {\em\JournalTitle{Proceedings of the American Mathematical Society}},
	137(4):1411--1420, 2009.
	
	\bibitem{rempe_hausdorff_2009}
	Lasse Rempe and Gwyneth~M. Stallard.
	\newblock Hausdorff dimensions of escaping sets of transcendental entire
	functions.
	\newblock {\em\JournalTitle{Proceedings of the American Mathematical Society}},
	138(05):1657--1665, December 2009.
	
	\bibitem{rempe-gillen_non-autonomous_2016}
	Lasse Rempe-Gillen and Mariusz Urba\'nski.
	\newblock Non-autonomous conformal iterated function systems and {Moran}-set
	constructions.
	\newblock {\em\JournalTitle{Transactions of the American Mathematical
			Society}}, 368(3):1979--2017, 2016.
	
	\bibitem{roy_random_2011}
	Mario Roy and Mariusz Urba\'nski.
	\newblock Random graph directed {Markov} systems.
	\newblock {\em\JournalTitle{Discrete Contin. Dyn. Syst}}, 30(1):261--298, 2011.
	
	\bibitem{shishikura_boundary_1994}
	Mitsuhiro Shishikura.
	\newblock The boundary of the {Mandelbrot} set has {Hausdorff} dimension two.
	\newblock {\em\JournalTitle{Ast\'erisque}}, (222):7, 389--405, 1994.
	
	\bibitem{skorulski_existence_2006}
	Bartlomiej Skorulski.
	\newblock The existence of conformal measures for some transcendental
	meromorphic functions.
	\newblock {\em\JournalTitle{Complex dynamics}}, 396:169--201, 2006.

\end{thebibliography}

\end{document}